\documentclass[11pt]{amsart}

\usepackage{graphicx}

\usepackage{array}

\usepackage{amsmath,amsthm,amssymb}
\usepackage{aligned-overset}

\usepackage{color}
\usepackage{bm, bbm}

\usepackage{tikz}
\usetikzlibrary{arrows.meta} 

\usepackage{booktabs}

\usepackage{comment}

\usepackage{thmtools}
\usepackage{hyperref}
\usepackage{cleveref}
\usepackage{placeins}

\theoremstyle{plain}
\newtheorem{theorem}{Theorem}[section]

\newtheorem{lemma}[theorem]{Lemma}

\newtheorem{proposition}[theorem]{Proposition}

\theoremstyle{definition}
\newtheorem{definition}[theorem]{Definition}

\newtheorem{example}[theorem]{Example}

\newtheorem{remark}[theorem]{Remark}

\numberwithin{equation}{section}

\newcommand{\baroplus}{\mathop{\overline{\bigoplus}}}

\renewcommand{\ker}{\mathrm{Ker}\,}

\title{Relation between generalized and ordinary cluster algebras}
\author[R. Akagi]{Ryota Akagi}

\author[T. Nakanishi]{Tomoki Nakanishi}

\address{Graduate School of Mathematics, Nagoya University, Chikusa-ku, Nagoya, 464-8604, Japan}
\email{ryota.akagi.e6@math.nagoya-u.ac.jp}
\email{nakanisi@math.nagoya-u.ac.jp}

\begin{document}

\begin{abstract}
Recently, Ramos and Whiting showed that any generalized cluster algebra of geometric type is isomorphic to a quotient of a subalgebra of a certain cluster algebra.
Based on their idea and method, we show that the same property holds for any generalized cluster algebra with $y$-variables in an arbitrary semifield. We also present the relations between the $C$-matrices, the $G$-matrices, and the $F$-polynomials of a generalized cluster pattern and those of the corresponding composite cluster pattern.
\end{abstract}

\maketitle

\section{Introduction}

Cluster algebras were introduced by Fomin and Zelevinsky \cite{FZ02}.
They naturally appear in several areas of mathematics and have been studied and developed by many authors. In \cite{CS14}, Chekhov and Shapiro introduced a natural generalization of cluster algebras called {\em generalized cluster algebras} (GCAs) by replacing the binomial exchange polynomials of (ordinary) cluster algebras (CAs) with the multinomial ones.
It turned out that GCAs share several important properties with CAs, for example, the Laurent phenomenon \cite{CS14}, the separation formulas \cite{Nak15a}, quantizations \cite{Nak15b, BCMX18, FPY24}, scattering diagram formulation \cite{CKM21, Mou24}, the Laurent positivity \cite{Mou24, BLM25a, BLM25b}, etc.
\par
Suggested by the question in \cite[\S5]{CS14}, Ramos and Whiting \cite{RW25} recently showed that any GCA $\mathcal{A}^{\mathrm{g}}$ \emph{of geometric type} (i.e., with coefficients in a tropical semifield) is isomorphic to a \emph{subquotient} of a certain CA $\mathcal{A}$, that is, a quotient of a subalgebra of $\mathcal{A}$. This identification finally reveals the true nature of GCAs.
Let us summarize the main idea in \cite{RW25} without explaining details.
Firstly, the exchange matrices of $\mathcal{A}$ are given by certain  extensions (``enlargements'') of the exchange matrices of $\mathcal{A}^{\mathrm{g}}$ so that the ``composite mutations'' of the enlarged exchange matrices are compatible with the generalized mutations of the exchange matrices of $\mathcal{A}^{\mathrm{g}}$.
Secondly, we consider the product of the cluster variables in $\mathcal{A}$ mutated by each composite mutation. We restrict our attention to the {\em subalgebra} $\mathcal{A}^{\mathrm{c}}$ of $\mathcal{A}$ generated by these cluster monomials.
Lastly, we impose a certain relation on $\mathcal{A}^{\mathrm{c}}$ so that the exchange polynomials of $\mathcal{A}^{\mathrm{g}}$ coincide with the exchange polynomials of these cluster monomials in $\mathcal{A}^{\mathrm{c}}$. (This implies the factorization of the exchange polynomials of $\mathcal{A}^{\mathrm{g}}$ into binomials under the relation.)
Then, $\mathcal{A}^{\mathrm{g}}$ is isomorphic to the {\em quotient} of $\mathcal{A}^{\mathrm{c}}$ by the relation ideal.
\par
Due to the importance of the above result, in this paper we apply the same idea and method to GCAs with \emph{coefficients in an arbitrary semifield} and obtain a parallel result for them.
To do it, in contrast to \cite{RW25}, we concentrate on the class of GCAs \emph{with $y$-variables} in \cite{Nak15a, NR16, Nak24}, which is parallel to the one of CAs in \cite{FZ07}.
Compared with the formulation of GCAs in \cite{RW25}, the advantage of ours is two-folded.
Firstly, we are away from the specific problem arising from the exchange matrix presentation of tropical coefficients  in  \cite{RW25}.
(A considerable part of the paper \cite{RW25} is devoted to resolve the issues related to it.)
Secondly, the parallelism of the factorization mechanism between cluster variables ($x$-variables) and coefficients ($y$-variables) is transparent. We also present the relations between the $C$-matrices, the $G$-matrices, and the $F$-polynomials of a generalized cluster pattern and those of the corresponding composite cluster pattern.
\par
Though mostly our method is a straightforward extension of the one in  \cite{RW25}, there are some subtleties which we need to take care of.
As explained above, to obtain the desired factorization of the exchange polynomials for the $x$-variables of a GCA $\mathcal{A}^{\mathrm{g}}$, we need to impose a certain relation on the subalgebra $\mathcal{A}^{\mathrm{c}}$ of the corresponding CA $\mathcal{A}$.
The relation ideal of $\mathcal{A}^{\mathrm{c}}$ is prime; thus, its quotient is an integral domain.
(This is consistent with the fact that the quotient is isomorphic to $\mathcal{A}^{\mathrm{g}}$, which is an integral domain.)
However, if we extend the relation to $\mathcal{A}$, the relation ideal of $\mathcal{A}$ might be non-prime so that its quotient is not an integral domain, in general.
A similar situation occurs also for $y$-variables, and we need to handle it carefully.
In fact, clarifying these details is one of the main achievement of the paper. The summary of this point is given in the diagram \eqref{eq: diagram}.
\par
The content of the paper is as follows.
In \Cref{sec: definition}, we recall the definition of generalized cluster algebras with $y$-variables. In \Cref{sec: composite cluster patterns}, we introduce composite mutations and composite cluster patterns in some ordinary cluster algebras. In \Cref{sec: Y-pattern}, we realize the generalized $Y$-patterns by the composite ones. In \Cref{sec: cluster algebras}, we realize the generalized cluster patterns and the generalized cluster algebras by the composite ones. In \Cref{sec: separation formulas}, we present the relations between the $C$-matrices, the $G$-matrices and the $F$-polynomials of generalized cluster algebras and those of composite ones.
\par
We thank Zhichao Chen and Yasuaki Gyoda for discussions. We thank Woojin Choi for pointing out the error in the early version of the manuscript.
This work is supported in part by JSPS grant No. JP22H01114 and No. JP25KJ1438. Ryota Akagi is also supported by Chubei Itoh Foundation.

\section{Generalized cluster algebras with $y$-variables}\label{sec: definition}
In this section, we recall the definition of generalized cluster algebras (GCAs) with $y$-variables following \cite{Nak15a, NR16, Nak24}.

\subsection{Semifields}\label{sec: semifield}
Following \cite{FZ02}, we recall the definition of a semifield.
\begin{definition}[Semifield]
A triplet $(\mathbb{P},\oplus,\cdot)$ is called a {\em semifield} if it satisfies the following conditions.
\begin{itemize}
\item $\mathbb{P}$ is an abelian group with respect to the multiplication $\cdot$.
\item $\mathbb{P}$ is a commutative semigroup with respect to the addition $\oplus$.
\item The distributive property $p(q_1\oplus q_2)=pq_1 \oplus pq_2$ holds for any $p,q_1,q_2 \in \mathbb{P}$.
\end{itemize}
For any semifield $\mathbb{P}$, we set $\mathbb{P}_0=\mathbb{P} \sqcup \{0\}$, and we extend the addition and the multiplication to $\mathbb{P}_0$ by $0 \oplus p = p \oplus 0 = p$ and $0 \cdot p = p \cdot 0 = 0$ for any $p \in \mathbb{P}_0$. Since $\mathbb{P}_0$ is not an abelian group with respect to $\cdot$, it is not a semifield in the above sense.
\end{definition}
Let $\mathbb{ZP}$ be the group ring of $(\mathbb{P},\cdot)$ over $\mathbb{Z}$.
We set
\begin{equation}\label{eq: set NP}
\mathbb{NP}=\left\{\left.\ \sum_{j=1}^{r} a_{j}p_{j} \in \mathbb{ZP}\ \right|\ r \in \mathbb{Z}_{\geq 0},\ a_{j} \in \mathbb{Z}_{\geq 0},\ p_j \in \mathbb{P} \right\},
\end{equation}
and we define an additive and multiplicative homomorphism $\pi_{\mathbb{P}_0}  :  \mathbb{NP}  \rightarrow  \mathbb{P}_0$ by
\begin{equation}\label{eq: specialzation for coefficients}
\pi_{\mathbb{P}_0}\left(\sum_{j=1}^{r}a_jp_j\right)  =  \bigoplus_{j=1}^{r}a_{j\oplus}p_j \quad (a_j \in \mathbb{Z}_{\geq 0}, p_j \in \mathbb{P}),
\end{equation}
where $a_{\oplus}p=p \oplus p \oplus \cdots \oplus p$ ($a \in \mathbb{Z}_{\geq 0}$, $p \in \mathbb{P}$) is the sum of $a$ terms of $p$ with respect to the addition $\oplus$ in $\mathbb{P}$. This should be distinguished from $ap=p+p+\cdots+p$, which is the sum in $\mathbb{ZP}$.
\par
The ring $\mathbb{ZP}$ is an integral domain \cite{FZ02}. The fraction field of $\mathbb{ZP}$ is denoted by $\mathbb{QP}$.
\subsection{Generalized seeds and GCAs}
When we consider a GCA with $y$-variables, we always fix the following data:
\begin{itemize}
\item a positive integer $n \in \mathbb{Z}_{\geq 1}$ called the {\em rank}.
\item an $n$-tuple of positive integers ${\bf r}=(r_1,\dots,r_n) \in \mathbb{Z}_{\geq 1}^n$ called the {\em mutation degree}.
\item a semifield $\mathbb{P}$ called the {\em coefficient semifield}.
\item a field of rational functions $\mathcal{F}$ in $n$-variables with coefficients in $\mathbb{QP}$ called the {\em ambient field}.
\end{itemize}

\begin{definition}
A quadruplet $\Sigma=({\bf x},{\bf y},{\bf Z},B)$ is called a {\em generalized seed} if it satisfies the following conditions.
\begin{itemize}
\item ${\bf x}=(x_1,\dots,x_n)$ is an $n$-tuple of elements in $\mathcal{F}$ forming a free generating set. Each $x_i$ is called a {\em cluster variable} ({\em $x$-variable}).
\item ${\bf y}=(y_1,\dots,y_n)$ is an $n$-tuple of elements in $\mathbb{P}$. Each $y_i$ is called a {\em coefficient} ({\em $y$-variable}).
\item ${\bf Z}=(Z_1(u),\dots,Z_n(u))$ is an $n$-tuple of polynomials $Z_i(u) \in \mathbb{NP}[u]$ of the form
\begin{equation}
Z_{i}(u)=1+\hat{z}_{i1}u+\hat{z}_{i2}u^2+\cdots+\hat{z}_{i,r_{i}-1}u^{r_i-1}+u^{r_i} \quad (\hat{z}_{il} \in \mathbb{NP}).
\end{equation}
Each $Z_{i}(u)$ is called an {\em exchange polynomial}. We allow some of $\hat{z}_{il}$ to be zero. We set $\hat{z}_{i0}=\hat{z}_{ir_i}=1$.
\item $B=(b_{ij}) \in \mathrm{Mat}_n(\mathbb{Z})$ is a skew-symmetrizable matrix, that is, there is a rational diagonal matrix $D=\mathrm{diag}(d_1,\dots,d_n)$ with $d_i > 0$ such that $DB$ is skew-symmetric. It is called an {\em exchange matrix}.
\end{itemize}
For a generalized seed $\Sigma$, we define {\em $\hat{y}$-variables} as
\begin{equation}
\hat{y}_k=y_k\prod_{j=1}^n x_j^{{b}_{jk}} \quad (k=1,\dots,n).
\end{equation}
\par
For a generalized seed $({\bf x},{\bf y},{\bf Z},B)$, we sometimes focus on the $y$-variables ${\bf y}$ only. In this case, the triplet $({\bf y},{\bf Z},B)$ is called a {\em generalized $Y$-seed}.
\end{definition}
For an exchange polynomial $Z_k(u)=\sum_{l=0}^{r_k}\hat{z}_{kl}u^l$ and $y \in \mathbb{P}$, we define the {\em specialization} $Z_{k}|_{\mathbb{P}_0}(y)$ of $Z_k(u)$ at $y$ by
\begin{equation}\label{eq: specialization of exchange polynomials}
Z_k|_{\mathbb{P}_0}(y)=\bigoplus_{l=0}^{r_k}\pi_{\mathbb{P}_0}(\hat{z}_{kl})y^l.
\end{equation}
Since $\hat{z}_{k0}=\hat{z}_{kr_k}=1$, we have $Z_{k}|_{\mathbb{P}_0}(y) \in \mathbb{P}$.
We write $\pi_{\mathbb{P}_0}(\hat{z}_{kl}) \in \mathbb{P}_0$ as $z_{kl}$, for simplicity.
\begin{definition}
Let $\Sigma=({\bf x},{\bf y},{\bf Z},B)$ be a generalized seed. Then, the {\em mutation} $\mu_{k}^{\bf r}(\Sigma)=({\bf x}',{\bf y}',{\bf Z}',B')$ in direction $k=1,2,\dots,n$ with the mutation degree ${\bf r}$ is defined as follows:
\begin{itemize}
\item ${\bf x}'=(x_1,\dots,x_k',\dots,x_n)$, where
\begin{equation}\label{eq: mutation of x}
x_k'=x_k^{-1}\left(\prod_{j=1}^{n}x_j^{[-b_{jk}]_{+}}\right)^{r_k}\frac{Z_{k}(\hat{y}_k)}{Z_{k}|_{\mathbb{P}_0}(y_k)}.
\end{equation}
\item ${\bf y}'=(y_1',\dots,y_n')$, where
\begin{equation}\label{eq: mutation of y}
y_i'=\begin{cases}
y_k^{-1} & i=k,\\
y_i\left(y_k^{[b_{ki}]_{+}}\right)^{r_k}Z_{k}|_{\mathbb{P}_0}(y_k)^{-b_{ki}} & i\neq k.
\end{cases}
\end{equation}
\item ${\bf Z}'=(Z_1(u),\dots,Z_k^{\mathrm{rec}}(u),\cdots,Z_n(u))$, where $Z_{k}^{\mathrm{rec}}(u)=u^{r_k}Z_k(u^{-1})$ is the reciprocal polynomial of $Z_{k}(u)$.
\item $B'=(b'_{ij})= \mu^{\bf r}_k(B)$, where
\begin{equation}\label{eq: B-mutation}
b'_{ij}=\begin{cases}
-b_{ij} & i = k\ \textup{or}\ j=k,\\
b_{ij}+r_k(b_{ik}[b_{kj}]_{+}+[-b_{ik}]_{+}b_{kj}) & i,j \neq k.
\end{cases}
\end{equation}
\end{itemize}
\end{definition}
Let $\mathbb{T}_n$ be the $n$-regular tree with edges labeled by $1,2,\dots,n$. We say that a pair of vertices $t,t' \in \mathbb{T}_n$ are {\em $k$-adjacent} ($k=1,\dots,n$) if $t$ and $t'$ are connected with an edge labeled by $k$. Let us fix a given vertex (the {\em initial vertex}) $t_0 \in \mathbb{T}_n$. Then, the collection of generalized seeds ${\bf \Sigma}=\{\Sigma_t\}_{t \in \mathbb{T}_n}$ is uniquely determined by the {\em initial seed} $\Sigma_{t_0}$ and the condition $\Sigma_{t'}=\mu_k^{\bf r}(\Sigma_{t})$ for any pair of $k$-adjacent vertices $t,t' \in \mathbb{T}_n$. The collection ${\bf \Sigma}$ is called the {\em generalized cluster pattern} with the initial seed $\Sigma_{t_0}$. The collection of the generalized $Y$-seeds ${\bf \Upsilon}=\{({\bf y}_t,{\bf Z}_t,B_t)\}_{t \in \mathbb{T}_n}$ of ${\bf \Sigma}$ is called the {\em generalized $Y$-pattern} of ${\bf \Sigma}$.

\begin{definition}[Generalized cluster algebra]
Let ${\bf \Sigma}$ be a generalized cluster pattern.
Let $\mathcal{X}$ be the set of all cluster variables appearing in ${\bf \Sigma}$. Then, the {\em generalized cluster algebra} $\mathcal{A}({\bf \Sigma})$ of ${\bf \Sigma}$ is defined by the $\mathbb{ZP}$-subalgebra $\mathbb{ZP}[\mathcal{X}]$ of $\mathcal{F}$ generated by $\mathcal{X}$.
\end{definition}
When ${\bf r}=(1,1,\dots,1)$, a generalized cluster pattern and a GCA reduce to an {\em ordinary cluster pattern} and an {\em ordinary cluster algebra} (CA) in \cite{FZ07}. In this case, its exchange polynomials are given by $Z_{i}(u)=1+u$. So, we omit ${\bf Z}$ and write an ordinary seed as $\Sigma=({\bf x},{\bf y},B)$. The mutation $\mu_k^{\bf r}$ also reduces to the ordinary mutation $\mu_{k}$ in \cite{FZ07}.

\section{Composite cluster patterns}\label{sec: composite cluster patterns}
In this section, following \cite{RW25}, for a generalized cluster pattern, we introduce a corresponding {\em composite cluster pattern} whose mutations are given by the {\em composite mutations}.
\subsection{Enlargements of exchange matrices}
For the mutation degree ${\bf r}=(r_1,\dots,r_n)$ of a generalized seed $\Sigma$, the sum $\mathcal{N}=r_1+\cdots+r_n$ is called the {\em pseudo-rank} of $\Sigma$, which will be the rank of a corresponding ordinary seed.
Let $\mathbbm{1}^{l \times m}$ be the $l \times m$ matrix whose all components are $1$.
\begin{definition}[Enlargement]
\label{def: canonical extension}
Let $B$ be the exchange matrix of a given generalized seed with mutation degree ${\bf r}$. We define an $\mathcal{N} \times \mathcal{N}$ matrix $\mathcal{B} \in \mathrm{Mat}_{\mathcal{N}}(\mathbb{Z})$ with the $n \times n$ block decomposition whose $(i,j)$th block $\mathcal{B}^{ij}$ is given by
\begin{equation}\label{eq: definition of enlargements}
\mathcal{B}^{ij}=b_{ij}\mathbbm{1}^{r_{i}\times r_j}.
\end{equation}
We call $\mathcal{B}$ the {\em enlargement} of $B$ with respect to ${\bf r}$. Let $\mathcal{B}^{ij}_{lm}$ denote the $(l,m)$th entry of the $(i,j)$th block $\mathcal{B}^{ij}$. 
\end{definition}
\begin{example}\label{exam: enlargement}
When ${\bf r}=(2,3)$ and $B=\left(\begin{smallmatrix}
0 & 1\\
-2 & 0
\end{smallmatrix}\right)$, the enlargement of $B$ is given by
\begin{equation}\label{eq: example of an enlargement}
\mathcal{B}=\left(\begin{array}{cc|ccc}
0 & 0 & 1 & 1 & 1\\
0 & 0 & 1 & 1 & 1\\
\hline
-2 & -2 & 0 & 0 & 0\\
-2 & -2 & 0 & 0 & 0\\
-2 & -2 & 0 & 0 & 0
\end{array}\right).
\end{equation}
\end{example}
Let us define the index set
\begin{equation}\label{eq: D}
\mathcal{D}=\{(i,l) \mid i=1,\dots,n;\ l=1,\dots,r_i\}.
\end{equation}
We write $\mathcal{B}^{ij}_{lm}$ also as $\mathcal{B}_{il,jm}$ ($(i,l),(j,m) \in \mathcal{D}$).
We consider an ordinary seed whose exchange matrix is $\mathcal{B}$. Then, for each $(i,l) \in \mathcal{D}$, the mutation in direction $r_1+r_2+\cdots+r_{i-1}+l$ is denoted by $\mu_{il}$ in accordance with the above parametrization of $\mathcal{B}$.
\begin{definition}[Composite mutation]
\label{def: composite mutation}
For each $k=1,2,\dots,n$, we define the {\em composite mutation} $\mu^{\mathrm{c}}_{k}$ for an ordinary seed $({\bf x},{\bf y},\mathcal{B})$ by
\begin{equation}\label{eq: definition of composite mutations}
\mu^{\mathrm{c}}_{k}=\mu_{kr_k} \circ \cdots \circ \mu_{k2} \circ \mu_{k1}.
\end{equation}
\end{definition}
The composite mutation $\mu^{\mathrm{c}}_k$ is independent of the order of the composition thanks to $\mathcal{B}^{kk}=O$. 
Moreover, the composite mutations are compatible with the enlargement as follows.
\begin{lemma}[{\cite[Lem.~22]{RW25}}]\label{lem: canonical extension}
For a generalized cluster pattern with mutation degree ${\bf r}$, let $\{B_{t}\}_{t \in \mathbb{T}_n}$ be the collection of its exchange matrices. Let $\mathcal{B}_{t}$ be the enlargement of $B_t$ with ${\bf r}$. Then, for any pair of $k$-adjacent vertices $t,t' \in \mathbb{T}_n$, we have
\begin{equation}
\mathcal{B}_{t'}=\mu^{\mathrm{c}}_k(\mathcal{B}_t).
\end{equation}
\end{lemma}
\begin{proof}
Let $\mathcal{B}$ be the enlargement of $B=(b_{ij})$. Let $B'=(b_{ij}')=\mu_k^{\bf r}(B)$ and $\mathcal{B}'=\mu_k^{\mathrm{c}}(\mathcal{B})$. Then, for $(i,l),(j,m) \in \mathcal{D}$ with $i,j \neq k$, we have
\begin{equation}
\begin{aligned}
(\mathcal{B}')^{ij}_{lm}&=\mathcal{B}^{ij}_{lm}+\sum_{p=1}^{r_k}(\mathcal{B}^{ik}_{lp}[\mathcal{B}^{kj}_{pm}]_{+}+[-\mathcal{B}^{ik}_{lp}]_{+}\mathcal{B}^{kj}_{pm})\\
&=b_{ij}+r_k(b_{ik}[b_{kj}]_{+}+[-b_{ik}]_{+}b_{kj})=b'_{ij}.
\end{aligned}
\end{equation}
For $i=k$ or $j=k$, we have $(\mathcal{B}')^{ij}_{lm}=-\mathcal{B}^{ij}_{lm}=-b_{ij}=b_{ij}'$ because $\mathcal{B}^{kk}=O$.
\end{proof}

\subsection{Composite cluster patterns}
Let $\mathcal{B}$ be the matrix in \eqref{eq: definition of enlargements}. In accordance with the parametrization of $\mathcal{B}$, when we consider an ordinary seed $\Sigma=({\bf x},{\bf y},\mathcal{B})$, we parametrize ${\bf x}$ and ${\bf y}$ by $\mathcal{D}$ as ${\bf x}=(x_{il})_{(i,l) \in \mathcal{D}}$ and ${\bf y}=(y_{il})_{(i,l) \in \mathcal{D}}$.
Then, we apply the composite mutations $\mu^{\mathrm{c}}_k$ in \eqref{eq: definition of composite mutations} to $\Sigma$. Explicitly, we have the following formula.
\begin{lemma}\label{lem: composite mutation for each vertices}
For an ordinary seed $\Sigma=({\bf x},{\bf y},\mathcal{B})$, the composite mutation $\mu^{\mathrm{c}}_k(\Sigma)=({\bf x}',{\bf y}',\mu^{\mathrm{c}}_k(\mathcal{B}))$ in direction $k=1,\dots,n$ is given as follows:
\begin{align}
\label{eq: composite x-mutation}
x'_{il} &= \begin{cases}
\displaystyle{x_{kl}^{-1}\biggl(\prod_{(j,m) \in \mathcal{D}}x_{jm}^{[-{b}_{jk}]_{+}}\biggr)\frac{1+\hat{y}_{kl}}{1 \oplus y_{kl}}} & i=k,\\
x_{il} & i \neq k.
\end{cases}
\\
\label{eq: composite y-mutation}
y'_{il} &= \begin{cases}
y_{kl}^{-1} & i=k,\\
\displaystyle{y_{il}\left(\prod_{m=1}^{r_k} y_{km}\right)^{[{b}_{ki}]_{+}}\left(\prod_{m=1}^{r_k}(1\oplus y_{km})\right)^{-{b}_{ki}}} & i \neq k.
\end{cases}
\end{align}
\end{lemma}
\begin{proof}
They follow from $\mathcal{B}^{ij}_{lm}=b_{ij}$, \eqref{eq: mutation of x}, and \eqref{eq: mutation of y} with ${\bf r}=(1,\dots,1)$.
\end{proof}

\begin{definition}[Composite cluster pattern]
In the definition of an ordinary cluster pattern, we replace the mutations $\mu_k$ ($k=1,\dots,\mathcal{N}$) with the composite ones $\mu^{\mathrm{c}}_k$ ($k=1,\dots,n$). We call the resulting collection ${\bf \Sigma}^{\mathrm{c}}=\{\Sigma^{\mathrm{c}}_t\}_{t \in \mathbb{T}_n}$ of ordinary seeds the {\em composite cluster pattern} with the initial seed $\Sigma_{t_0}^{\mathrm{c}}$. This is a subcollection of the ordinary cluster pattern ${\bf \Sigma}$ of rank $\mathcal{N}$ with the same initial seed $\Sigma_{t_0}=\Sigma_{t_0}^{\mathrm{c}}$. We also define the {\em composite $Y$-pattern} ${\bf \Upsilon}^{\mathrm{c}}$ of ${\bf \Sigma}^{\mathrm{c}}$ as the collection of the $Y$-seeds in ${\bf \Sigma}^{\mathrm{c}}$.
\end{definition}

\section{Realization of generalized $Y$-pattern}\label{sec: Y-pattern}
In this section, for a given generalized $Y$-pattern ${\bf \Upsilon}^{\mathrm{g}}$, we construct a composite $Y$-pattern ${\bf \Upsilon}^{\mathrm{c}}$ such that ${\bf \Upsilon}^{\mathrm{g}}$ is recovered from ${\bf \Upsilon}^{\mathrm{c}}$.
\subsection{Extensions of semifields}
We fix a generalized seed $\Sigma=({\bf x},{\bf y},{\bf Z},B)$ with coefficients in $\mathbb{P}$ and mutation degree ${\bf r}$. Let us introduce formal variables $s_{il}$ ($(i,l) \in \mathcal{D}$) and set ${\bf s}=\{s_{il} \mid (i,l) \in \mathcal{D}\}$. Since $\mathbb{ZP}$ is an integral domain, the polynomial ring $\mathbb{ZP}[{\bf s}]$ in ${\bf s}$ is also so. Thus, we have the fraction field $\mathbb{QP}({\bf s})$ of $\mathbb{ZP}[{\bf s}]$. Let $\mathbb{NP}[{\bf s}] \subset \mathbb{ZP}[{\bf s}]$ be the set of all polynomials in ${\bf s}$ such that all coefficients belong to $\mathbb{NP}$ in \eqref{eq: set NP}. We introduce the {\em universal semifield} $\mathbb{QP}_{\mathrm{sf}}({\bf s})$ with coefficients in $\mathbb{P}$ by
\begin{equation}
\mathbb{QP}_{\mathrm{sf}}({\bf s})=\left\{\left. \frac{f}{g} \in \mathbb{QP}({\bf s})\ \right| f,g \in \mathbb{NP}[{\bf s}]\setminus\{0\}\right\}.
\end{equation}
This is a semifield with respect to the addition and multiplication of fractions in $\mathbb{QP}({\bf s})$. Below, we write $\mathbb{QP}_{\mathrm{sf}}({\bf s})$ as $\overline{\mathbb{P}}$, for simplicity.
\par
Note that the addition in $\overline{\mathbb{P}}$ is different from the one in $\mathbb{P}$. When we consider cluster algebras in \Cref{sec: cluster algebras}, we need to further consider the group rings $\mathbb{ZP}$ and $\mathbb{Z}\overline{\mathbb{P}}$, whose additions are also different. To distinguish them, we write these additions as follows:
\begin{itemize}
\item $\oplus$ is the addition in $\mathbb{P}$.
\item $\bar{\oplus}$ is the addition in $\overline{\mathbb{P}}$.
\item $+$ is the addition in $\mathbb{ZP}$ and $\mathbb{Z}\overline{\mathbb{P}}$.
\end{itemize}
Note that $\mathbb{NP}$ is a subset of both $\overline{\mathbb{P}}$ and $\mathbb{ZP}$. Since the additions $\bar{\oplus}$ and $+$ coincide in $\mathbb{NP}$, we use both symbols depending on the context.
\par
We would like to impose an equivalence relation $\sim$ on $\overline{\mathbb{P}}_0$,
which is compatible with the addition and the multiplication, so that the following relation holds:
\begin{equation}\label{eq: factorization}
Z_i|_{\overline{\mathbb{P}}_0}(y) \sim (1 \bar{\oplus}s_{i1}y)(1 \bar{\oplus}s_{i2}y)\cdots(1 \bar{\oplus} s_{ir_i}y).
\end{equation}
However, imposing such a relation directly on $\overline{\mathbb{P}}_0$ causes possible problems as follows.
\par
(1)
There is no guarantee that the resulting relation satisfies
the primeness condition, i.e., $r_1 r_2 \sim 0$ implies $r_1 \sim 0$ or $r_2 \sim 0$.
And, if the primeness condition does not hold,  we have $1 \sim 0$, so that the quotient  $\overline{\mathbb{P}}_0/{\sim}$
collapses to $\{ 0\}$.
\par
(2) Even if the primeness condition holds,
 there is no guarantee that the map $\mathbb{P} \rightarrow \overline{\mathbb{P}}_0/{\sim}$, $p \mapsto [p]$ is injective, in general.
\par 
To avoid these problems, for each $i=1,\dots,n$, we introduce the elementary symmetric polynomial in $s_{i1}$,\dots,$s_{ir_i}$ of degree $l=1,\dots,r_i$ as
\begin{equation}\label{eq: idea of the construction}
e_{il} = \baroplus_{1 \leq m_1<m_2<\cdots<m_l \leq r_i}s_{im_1}s_{im_2}\cdots s_{im_l}  \in \overline{\mathbb{P}}.
\end{equation}
Let ${\bf e}=\{{e}_{il} \mid (i,l) \in \mathcal{D}\}$. Let ${\bf Z}$ be the exchange polynomials of $\Sigma$. By extending the homomorphism $\pi_{\mathbb{P}_0}$ in \eqref{eq: specialzation for coefficients}, we define an additive and multiplicative homomorphism $\psi=\psi_{\bf Z}:\mathbb{NP}[{\bf e}] \to \mathbb{P}_0$ by
\begin{equation}
\psi(e_{il})=z_{il}.
\end{equation}
In particular, since $z_{ir_i}=\pi_{\mathbb{P}_0}(\hat{z}_{ir_i})=1$, we have
\begin{equation}\label{eq: highest order unitness}
\psi(e_{ir_i})=\psi(s_{i1}s_{i2}\cdots s_{ir_i})=1 \quad (i=1,\dots,n).
\end{equation}
We introduce the following subsemifield of $\overline{\mathbb{P}}$:
\begin{equation}\label{eq: definition of S for y}
\mathbb{S}=\left\{\left. \frac{f}{g} \in \overline{\mathbb{P}} \ \right|\ f,g \in \mathbb{NP}[{\bf e}],\ \psi(f), \psi(g) \neq 0  \right\}.
\end{equation}
We view $\mathbb{P} \subset \mathbb{S}$ by the natural identification $p=p/1$ ($p \in \mathbb{P}$).
We obtain the following semifield homomorphism (denoted by the same symbol $\psi$)
\begin{equation}\label{eq: definition of psi}
\begin{array}{rccccc}
\psi & : & \mathbb{S} & \to & \mathbb{P} &\\
 & & \displaystyle{\frac{f}{g}} & \mapsto & \displaystyle{\frac{\psi(f)}{\psi(g)}} & (f,g \in \mathbb{NP}[{\bf e}]).
\end{array}
\end{equation}
The map $\psi$ induces the equivalence relation in $\mathbb{S}$ such that
\begin{equation}
f \sim g \Longleftrightarrow \psi(f)=\psi(g) \quad (f,g \in \mathbb{S}).
\end{equation}
Let $[f]$ be the equivalence class of $f \in \mathbb{S}$, and let $\mathbb{S}/{\sim}$ be the quotient of $\mathbb{S}$ by $\sim$.
\begin{lemma}\label{lem: S tilde}
The quotient $\mathbb{S}/{\sim}$ is a semifield. Moreover,
the map $\mathbb{P} \to \mathbb{S}/{\sim}$ defined by $p \mapsto [p]$ is a semifield isomorphism.
\end{lemma}
\begin{proof}
The composition $\mathbb{P} \hookrightarrow \mathbb{S} \overset{\psi}{\to} \mathbb{P}$ is the identity map, where $\mathbb{P} \hookrightarrow \mathbb{S}$ is the canonical inclusion. Thus, $\psi$ is surjective. Then, by the fundamental theorem on homomorphisms, the claim holds.
\end{proof}
Now, we have the desired relation \eqref{eq: factorization} in $\mathbb{S}$.
\begin{lemma}\label{lem: factrization for y-seeds}
For any $y \in \mathbb{S}$ and $k=1,2,\dots,n$, the products $\prod_{l=1}^{r_k}(1 \bar{\oplus} s_{kl}y)$ and $\prod_{l=1}^{r_k}(1 \bar{\oplus} s_{kl}^{-1}y)$ belong to $\mathbb{S}$. Moreover, we have
\begin{equation}
\begin{aligned}
\prod_{l=1}^{r_k}(1 \bar{\oplus} s_{kl}y) \sim Z_k|_{\mathbb{S}_0}(y),
\quad
\prod_{l=1}^{r_k}(1 \bar{\oplus} s_{kl}^{-1}y) \sim Z_k^{\mathrm{rec}}|_{\mathbb{S}_0}(y).
\end{aligned}
\end{equation}
\end{lemma}
\begin{proof}
Let us prove the first relation. Since $\prod_{l=1}^{r_k}(1 \bar{\oplus} s_{kl}y)=1\bar{\oplus}\baroplus_{l=1}^{r_k}e_{kl}y^l$, we have
\begin{equation}\label{eq: expansion and apply psi in S}
\psi\left(\prod_{l=1}^{r_k}(1 \bar{\oplus} s_{kl}y)\right)=1\oplus \bigoplus_{l=1}^{r_k}z_{kl}\psi(y)^l.
\end{equation}
The right hand side has a term $1 \in \mathbb{P}$. So, it is nonzero. Thus, the product $\prod_{l=1}^{r_k}(1 \bar{\oplus} s_{kl}y)$ belongs to $\mathbb{S}$. The first relation also follows from \eqref{eq: expansion and apply psi in S}. The second relation is obtained from the first one by replacing $y$ with $y^{-1}$ and multiplying $y^{r_k}e_{kr_k}^{-1} \sim y^{r_k}$.
\end{proof}

\subsection{Realization of generalized $Y$-patterns}
From now on, we consider two kinds of patterns. The first one is the generalized cluster pattern ${\bf \Sigma}^{\mathrm{g}}=\{({\bf x}_t^{\mathrm{g}}, {\bf y}_{t}^{\mathrm{g}},{\bf Z}_t,B_t)\}_{t \in \mathbb{T}_n}$, where the initial seed $\Sigma_{t_0}^{\mathrm{g}}$ is given by $\Sigma$ in the previous section. The second one is a composite cluster pattern ${\bf \Sigma}^{\mathrm{c}}=\{({\bf x}_t^{\mathrm{c}},{\bf y}_{t}^{\mathrm{c}},\mathcal{B}_t)\}_{t \in \mathbb{T}_n}$ whose detail will be specified below.
\par
Here, we focus on the generalized $Y$-pattern ${\bf \Upsilon}^{\mathrm{g}}$ of ${\bf \Sigma}^{\mathrm{g}}$ with the initial generalized $Y$-seed $({\bf y}_{t_0}^{\mathrm{g}},{\bf Z}_{t_0},B_{t_0})$. We introduce the corresponding initial ordinary $Y$-seed $({\bf y}_{t_0}^{\mathrm{c}},\mathcal{B}_{t_0})$ of rank $\mathcal{N}$ with coefficients in $\overline{\mathbb{P}}$, where $\mathcal{B}_{t_0}$ is the enlargement of $B_{t_0}$ in \eqref{eq: definition of enlargements}, and the $y$-variables are given by
\begin{equation}\label{eq: definition of y seeds}
y_{il;t_0}^{\mathrm{c}}=s_{il}y_{i;t_0}^{\mathrm{g}} \quad ((i,l) \in \mathcal{D}).
\end{equation}
Then, we have the composite $Y$-pattern ${\bf \Upsilon}^{\mathrm{c}}$ with the initial $Y$-seed $({\bf y}_{t_0}^{\mathrm{c}},\mathcal{B}_{t_0})$.
\par
For each $i =1,\dots,n$ and $t \in \mathbb{T}_n$, we recursively define the sign $\sigma_{i;t} \in \{\pm 1\}$ as follows: Set $\sigma_{i;t_0}=1$ for any $i=1,\dots,n$. For any pair of $k$-adjacent vertices $t,t' \in \mathbb{T}_n$, let
\begin{equation}\label{eq: definition of sigma}
\sigma_{k;t'}=-\sigma_{k;t},
\quad
\sigma_{i;t'}=\sigma_{i;t}\quad (i \neq k).
\end{equation}
Note that we have $Z_{i;t}(u)=Z_{i;t_0}(u)$ if $\sigma_{i;t}=1$, and $Z_{i;t}(u)=Z_{i;t_0}^{\mathrm{rec}}(u)$ if $\sigma_{i;t}=-1$. Then, by \Cref{lem: factrization for y-seeds}, we have
\begin{equation}\label{eq: equivalence in P}
\prod_{l=1}^{r_k}(1 \bar{\oplus} s_{kl}^{\sigma_{k;t}}y) \sim Z_{k;t}|_{\mathbb{S}_0}(y) \quad (y \in \mathbb{S}).
\end{equation}
\par
Let $\mathbb{S}_1=\psi^{-1}(1)$, where $\psi$ is the one in \eqref{eq: definition of psi}. Then, the following fact holds.
\begin{lemma}[cf.~{\cite[Lem.~24]{RW25}}]\label{lem: y expression}
For any $(i,l) \in \mathcal{D}$ and $t \in \mathbb{T}_n$, we have
\begin{equation}\label{eq: ordinary y variables}
y_{il;t}^{\mathrm{c}}=s_{il}^{\sigma_{i;t}}y_{i;t}^{\mathrm{g}}\mathfrak{u}_{i;t},
\end{equation}
where $\mathfrak{u}_{i;t} \in \mathbb{S}_1$ is independent of $l=1,\dots,r_i$.
\end{lemma}
\begin{proof}
When $t=t_0$, the claim follows from \eqref{eq: definition of y seeds}.
Suppose that the claim holds for some $t \in \mathbb{T}_n$. Let $t' \in \mathbb{T}_n$ be $k$-adjacent to $t$. We calculate $y_{il;t'}^{\mathrm{c}}$ by the formula \eqref{eq: composite y-mutation}. When $i=k$, $y_{kl;t'}^{\mathrm{c}}=s_{kl}^{-\sigma_{k;t}}(y_{k;t}^{\mathrm{g}})^{-1}\mathfrak{u}_{k;t}^{-1}$ also satisfies the same property. Assume $i \neq k$. Then, we have
\begin{equation}\label{eq: y expression}
\begin{aligned}
y_{il;t'}^{\mathrm{c}}&=s_{il}^{\sigma_{i;t}}y_{i;t}^{\mathrm{g}}(y_{k;t}^{\mathrm{g}})^{r_k[b_{ki;t}]_{+}}Z_{k;t}|_{\mathbb{S}_0}(y_{k;t}^{\mathrm{g}})^{-b_{ki;t}}
\\
&\qquad \times \mathfrak{u}_{i;t}\mathfrak{u}_{k;t}^{r_k[b_{ki;t}]_{+}}\left(\prod_{m=1}^{r_k}s_{km}\right)^{\sigma_{k;t}[b_{ki;t}]_{+}}\left(\frac{\prod_{m=1}^{r_k}(1 \bar{\oplus} s_{km}^{\sigma_{k;t}}y_{k;t}^{\mathrm{g}}\mathfrak{u}_{k;t})}{Z_{k;t}|_{\mathbb{S}_0}(y_{k;t}^{\mathrm{g}})}\right)^{-b_{ki;t}}.
\end{aligned}
\end{equation}
Then, every factor on the second row of \eqref{eq: y expression} belongs to $\mathbb{S}_1$ due to \eqref{eq: highest order unitness} and \eqref{eq: equivalence in P}. Thus, the claim holds.
\end{proof}
\begin{lemma}
For any $i=1,\dots,n$ and $t \in \mathbb{T}_n$, the product $\prod_{l=1}^{r_i}y_{il;t}^{\mathrm{c}}$ belongs to $\mathbb{S}$.
\end{lemma}
\begin{proof}
This follows from \eqref{eq: ordinary y variables}.
\end{proof}
As a result, ${\bf \Upsilon}^{\mathrm{g}}$ is realized in $\mathbb{S}/{\sim} \cong \mathbb{P}$ by ${\bf \Upsilon}^{\mathrm{c}}$ as follows.
\begin{theorem}\label{thm: realization of y-seeds}
In $\mathbb{S}/{\sim}$, we have
\begin{equation}
[y_{i;t}^\mathrm{g}]=\left[\prod_{l=1}^{r_i}y_{il;t}^{\mathrm{c}}\right]^{\frac{1}{r_i}},
\end{equation}
where the right hand side is the unique $r_i$th root of $\left[\prod_{l=1}^{r_i}y_{il;t}^{\mathrm{c}}\right]$ in $\mathbb{S}/{\sim}$.
\end{theorem}
\begin{proof}
By \eqref{eq: highest order unitness} and \eqref{eq: ordinary y variables}, we have
\begin{equation}
\prod_{l=1}^{r_i}y_{il;t}^{\mathrm{c}} \sim (y_{i;t}^{\mathrm{g}})^{r_i}.
\end{equation}
Thus, $[y_{i;t}^\mathrm{g}]$ is an $r_i$th root of $[\prod_{l=1}^{r_i}y_{il;t}^{\mathrm{c}}]$.
Also, by \cite{FZ02}, every semifield is torsion-free as a multiplicative abelian group. Therefore, it is unique.
\end{proof}

\section{Realization of generalized cluster algebras}\label{sec: cluster algebras}
In this section, for a given GCA $\mathcal{A}^{\mathrm{g}}$, we construct a CA $\mathcal{A}$ whose subquotient $\mathcal{A}^{\mathrm{c}}/{\mathcal{I}}$ is isomorphic to $\mathcal{A}^{\mathrm{g}}$.
\subsection{Realization of generalized cluster patterns}
As in \Cref{sec: Y-pattern}, let ${\bf \Sigma}^{\mathrm{g}}=\{({\bf x}^{\mathrm{g}}_{t},{\bf y}^{\mathrm{g}}_{t},{\bf Z}_t,B_t)\}_{t \in \mathbb{T}_n}$ be a generalized cluster pattern with the initial seed $\Sigma^{\mathrm{g}}_{t_0}$. Let $\overline{\mathbb{P}}=\mathbb{QP}_{\mathrm{sf}}({\bf s})$, and let $\overline{\mathcal{F}}$ be the field of the rational functions in $\mathcal{N}=r_1+\cdots+r_n$ variables with coefficients in $\mathbb{Q}\overline{\mathbb{P}}$. 
We take an ordinary seed $\Sigma_{t_0}^{\mathrm{c}}=({\bf x}_{t_0}^{\mathrm{c}},{\bf y}_{t_0}^{\mathrm{c}},\mathcal{B}_{t_0})$ of rank $\mathcal{N}$ with the ambient field $\overline{\mathcal{F}}$ such that its $Y$-seed $({\bf y}_{t_0}^{\mathrm{c}},\mathcal{B}_{t_0})$ is given by \eqref{eq: definition of y seeds}. Then, we consider the composite cluster pattern ${\bf \Sigma}^{\mathrm{c}}=\{({\bf x}_{t}^{\mathrm{c}},{\bf y}_t^{\mathrm{c}},\mathcal{B}_t)\}_{t \in \mathbb{T}_{n}}$ with the initial seed $\Sigma_{t_0}^{\mathrm{c}}$, which is a subcollection of the ordinary cluster pattern ${\bf \Sigma}=\{({\bf x}_{t},{\bf y}_t,\mathcal{B}_t)\}_{t \in \mathbb{T}_{\mathcal{N}}}$ with the same initial seed $\Sigma_{t_0}=\Sigma_{t_0}^{\mathrm{c}}$.
\par
For any $t \in \mathbb{T}_n$, we set
\begin{equation}\label{eq: definition of X and Y}
X_{i;t}=\prod_{l=1}^{r_i}x_{il;t}^{\mathrm{c}},
\quad
\hat{Y}_{i;t}=y_{i;t}^{\mathrm{g}}\prod_{j=1}^{r_i}X_{j;t}^{b_{ji;t}}.
\end{equation}
Then, for any $(i,l) \in \mathcal{D}$ and $t \in \mathbb{T}_n$, $\hat{y}$-variables $\hat{y}_{il;t}^{\mathrm{c}}$ of $\Sigma_t^{\mathrm{c}}$ is given by
\begin{equation}\label{eq: y hat expression}
\hat{y}_{il;t}^{\mathrm{c}}=y_{il;t}^{\mathrm{c}}\prod_{(j,m) \in \mathcal{D}}(x_{jm;t}^{\mathrm{c}})^{\mathcal{B}_{jm, il;t}}
=y_{il;t}^{\mathrm{c}}\prod_{j=1}^{n}X_{j;t}^{b_{ji;t}}
\overset{\eqref{eq: ordinary y variables}}{=}s_{il}^{\sigma_{i;t}}\hat{Y}_{i;t}\mathfrak{u}_{i;t}.
\end{equation}
We write ${\bf X}=(X_{1;t_0},\dots,X_{n;t_0})$, for simplicity.
\par
In parallel to $\mathbb{S}$ in \eqref{eq: definition of S for y}, we introduce the $\mathbb{ZS}$-subalgebra $\mathcal{S} \subset \overline{\mathcal{F}}$ as follows. For each $(i,l) \in \mathcal{D}$, let
\begin{equation}\label{eq: elementary symmetric polynomial}
\hat{e}_{il}=\sum_{1 \leq m_1<m_2<\cdots<m_l \leq r_i} s_{im_1}s_{im_2}\cdots s_{im_l} \in \overline{\mathcal{F}}
\end{equation}
be the elementary symmetric polynomial in $s_{i1},\dots,s_{ir_i}$ of degree $l$ with respect to the addition $+$ in $\overline{\mathcal{F}}$. Let $\hat{\bf e}=\{\hat{e}_{il} \mid (i,l) \in \mathcal{D} \}$.
Let $\mathbb{QP}({\bf X})$ be the fraction field of $\mathbb{ZP}[{\bf X}]$. Then, the semifield homomorphism $\psi:\mathbb{S} \to \mathbb{P}$ in \eqref{eq: definition of psi} can be extended to a $\mathbb{ZP}$-algebra homomorphism $\hat{\psi}:\mathbb{ZS}[\hat{\bf e}, {\bf X}] \to \mathbb{ZP}[{\bf X}]$ by $X_{i;t_0} \mapsto X_{i;t_0}$ and $\hat{e}_{il} \mapsto \hat{z}_{il}$. Let
\begin{equation}
\mathcal{S}=\left\{\left. \frac{f}{g} \in \overline{\mathcal{F}}\ \right|\ f,g \in \mathbb{ZS}[\hat{\bf e}, {\bf X}],\ \hat{\psi}(g) \neq 0 \right\}.
\end{equation}
This is a $\mathbb{ZP}$-subalgebra of $\overline{\mathcal{F}}$. One can extend the above $\hat{\psi}$ to a $\mathbb{ZP}$-algebra homomorphism $\hat{\psi}: \mathcal{S} \to \mathbb{QP}({\bf X})$ by $\psi(f/g)=\psi(f)/\psi(g)$.
Then, the following fact holds.
\begin{lemma}\label{lem: field isomorphism}
The quotient $\mathcal{S}/{\ker \hat{\psi}}$ is a field. Moreover, the $\mathbb{ZP}$-algebra homomorphism $\mathbb{QP}({\bf X}) \to \mathcal{S}/{\ker \hat{\psi}}$ defined by $f \mapsto [f]$ is a field isomorphism.
\end{lemma}
\begin{proof}
The composition $\mathbb{QP}({\bf X}) \hookrightarrow \mathcal{S} \overset{\hat{\psi}}{\to} \mathbb{QP}({\bf X})$ is the identity map, where $\mathbb{QP}({\bf X}) \hookrightarrow \mathcal{S}$ is the canonical inclusion. Thus, $\hat{\psi}$ is surjective. Then, by the fundamental theorem on homomorphisms, the claim holds.
\end{proof}
Let $\equiv$ denote the equivalence in $\mathcal{S}$ modulo the ideal $\ker \hat{\psi}$. Now, we have an analogous relation \eqref{eq: idea of the construction} in $\mathcal{S}$.
\begin{lemma}[cf.~{\cite[Lem.~29]{RW25}}]\label{lem: factrization for +}
For any $u \in \mathcal{S}$ and $k=1,\dots,n$, the products $\prod_{l=1}^{r_k}(1 + s_{kl}u)$ and $\prod_{l=1}^{r_k}(1 + s_{kl}^{-1}u)$ belong to $\mathcal{S}$. Moreover, we have
\begin{equation}\label{eq: equivalence in F}
\prod_{l=1}^{r_k}(1 + s_{kl}^{\sigma_{k;t}}u) \equiv Z_{k;t}(u).
\end{equation}
\end{lemma}
\begin{proof}
They are proved in the same way as \Cref{lem: factrization for y-seeds} and \eqref{eq: equivalence in P}.
\end{proof}

\begin{lemma}\label{lem: X in S}
Every $X_{i;t}$ belongs to $\mathcal{S}$, and it holds that $\hat{\psi}(X_{i;t}) \neq 0$.
\end{lemma}
\begin{proof}
We prove the claim by the induction on $t \in \mathbb{T}_n$. For $t=t_0$, the claim holds by definition. Suppose that the claim holds for some $t \in \mathbb{T}_n$, and let $t' \in \mathbb{T}_n$ be $k$-adjacent to $t$. By \eqref{eq: composite x-mutation}, \eqref{eq: ordinary y variables}, and \eqref{eq: y hat expression}, we have
\begin{equation}\label{eq: mutation for X}
X_{k;t'}=\frac{1}{X_{k;t}}\left(\prod_{j=1}^{n}X_{j;t}^{[-b_{jk;t}]_{+}}\right)^{r_k}\frac{\prod_{l=1}^{r_k}(1 + s_{kl}^{\sigma_{k;t}}\hat{Y}_{k;t}\mathfrak{u}_{k;t})}{\prod_{l=1}^{r_k}(1 \bar{\oplus} s_{kl}^{\sigma_{k;t}}y_{k;t}^{\mathrm{g}}\mathfrak{u}_{k;t})}.
\end{equation}
By the assumption, the first and the second factors on the right hand side satisfy the claim. Also, by Lemmas~\ref{lem: factrization for y-seeds} and \ref{lem: factrization for +}, the third factor belongs to $\mathcal{S}$, and the image of it by $\hat{\psi}$ is nonzero.
\end{proof}
Thanks to \Cref{lem: field isomorphism}, we have the field isomorphism $\Phi: \mathcal{F}=\mathbb{QP}({\bf x}_{t_0}^{\mathrm{g}}) \to \mathcal{S}/{\ker} \hat{\psi}$ by
\begin{equation}
\Phi(x_{i;t_0}^{\mathrm{g}})=[X_{i;t_0}],
\quad
\Phi(p)=[p] \quad (p \in \mathbb{P}).
\end{equation}
Now, the generalized cluster pattern ${\bf \Sigma}^{\mathrm{g}}$ is realized by the composite cluster pattern ${\bf \Sigma}^{\mathrm{c}}$ as follows.
\begin{theorem}[{cf.~\cite[Thm.~30]{RW25}}]\label{thm: main theorem for patterns}
By the above isomorphism $\Phi:\mathcal{F} \to \mathcal{S}/\ker \hat{\psi}$, we have
\begin{equation}\label{eq: isomorphism between the cluster pattern}
y_{i;t}^{\mathrm{g}} \mapsto \left[\prod_{l=1}^{r_i}y_{il;t}^{\mathrm{c}}\right]^{\frac{1}{r_i}},
\quad
x_{i;t}^{\mathrm{g}} \mapsto [X_{i;t}].
\end{equation}
\end{theorem}
\begin{proof}
For the $y$-variables, we have already shown the claim in \Cref{thm: realization of y-seeds}. 
For the $x$-variables, if the cliam holds for some $t \in \mathbb{T}_n$, by \eqref{eq: definition of X and Y}, we have $\Phi(\hat{y}_{i;t}^{\mathrm{g}})=[\hat{Y}_{i;t}]$. Then, by applying \eqref{eq: equivalence in P} and \eqref{eq: equivalence in F} to \eqref{eq: mutation for X}, we obtain the claim by the induction on $t \in \mathbb{T}_n$.
\end{proof}

\subsection{Realization of GCAs}
Let $\mathcal{A}^{\mathrm{g}}=\mathcal{A}({\bf \Sigma}^{\mathrm{g}})$ be the GCA of ${\bf \Sigma}^{\mathrm{g}}$, and let $\mathcal{A}=\mathcal{A}({\bf \Sigma})$ be the CA of ${\bf \Sigma}$.
Consider the following $\mathbb{ZS}$-subalgebra of $\mathcal{A}$:
\begin{equation}
\mathcal{A}^{\mathrm{c}}=\mathbb{ZS}[X_{i;t} \mid i=1,\dots,n,\ t \in \mathbb{T}_n] \subset \mathcal{A}.
\end{equation}
By \Cref{lem: X in S}, we also have $\mathcal{A}^{\mathrm{c}} \subset \mathcal{S}$. Let $\mathcal{I}=\ker\hat{\psi} \cap \mathcal{A}^{\mathrm{c}}$, and consider its quotient $\mathcal{A}^{\mathrm{c}}/{\mathcal{I}}$.
\par
We set $\tilde{\mathbb{S}}=\mathbb{S}/{\sim}$. All relevant algebras and their relations are summarized in the following diagram.
\begin{align}\label{eq: diagram}
\begin{tikzpicture}
\node at (0.3,0) {$\overline{\mathcal{F}}=\mathbb{Q} \overline{ \mathbb{P}} (\mathbf{x}_{t_0})$};
\node at (-1.5,-1) {$\mathcal{S}$};
\node at (3,-1) {$\mathcal{A}=\mathbb{Z} \overline{ \mathbb{P}} [x_{i;t} \vert i, t]$};
\node at (-6,-2) {$\mathcal{F} =\mathbb{Q}\mathbb{P}(\mathbf{x}_{t_0}^{\mathrm{g}})$};
\node at (-3,-2) {$\mathcal{S}/\mathrm{Ker}\, \hat\psi$};
\node at (0,-2) {$\mathcal{A}^{\mathrm{c}}+\mathrm{Ker}\, \hat\psi$};
\node at (3,-2.05) {$\mathcal{A}^{\mathrm{c}}=\mathbb{Z}  \mathbb{S} [X_{i;t} \vert i, t]$};
\node at (-6,-3) {$\mathcal{A}^{\mathrm{g}}=\mathbb{Z} \mathbb{P} [x_{i;t}^{\mathrm{g}} \vert i, t]$};
\node at (-1.5,-3) {$\mathbb{Z}  \tilde{\mathbb{S}} [[X_{i;t}] \vert i, t]$};
\node at (3,-3) {$\mathcal{A}^{\mathrm{c}}/\mathcal{I}$};
\node at (-4.25,-2) {$\simeq$};
\node at (-3.5,-3) {$\simeq$};
\node at (1.25,-3) {$\simeq$};
\node [rotate=90] at (-6,-2.5) {$\subset$};
\node [rotate=90] at (3,-1.5) {$\subset$};
\node [rotate=135] at (1.5,-0.5) {$\subset$};
\node [rotate=45] at (-1,-0.5) {$\subset$};
\node [rotate=135] at (-1,-1.5) {$\subset$};
\node [rotate=135] at (-2.5,-2.5) {$\subset$};
\node  at (1.3,-2.02) {$\supset$};
\draw [->] (3,-2.4) -- (3,-2.7);
\draw [->] (-1.9,-1.3) -- (-2.2,-1.6);
\draw [->] (-0.3,-2.4) -- (-0.6,-2.7);
\end{tikzpicture}
\end{align}
The following is the main result, which is parallel to the result of \cite{RW25}.
\begin{theorem}[cf.~{\cite[Thm.~1]{RW25}}]\label{thm: main theorem for cluster algebras}
The $\mathbb{ZP}$-algebra homomorphism $\Psi:\mathcal{A}^{\mathrm{g}} \to \mathcal{A}^{\mathrm{c}}/\mathcal{I}$ defined by $x_{i;t}^{\mathrm{g}} \mapsto [X_{i;t}]$ and $p \mapsto [p]$ ($p \in \mathbb{P}$) is an isomorphism. Therefore, $\mathcal{A}^{\mathrm{g}}$ is isomorphic to a subquotient $\mathcal{A}^{\mathrm{c}}/\mathcal{I}$ of $\mathcal{A}$.
\end{theorem}
\begin{proof}
In the bottom line of \eqref{eq: diagram}, the first isomorphism $\mathcal{A} \cong \mathbb{Z}\tilde{\mathbb{S}}[[X_{i;t}]\mid i,t]$, $x_{i;t}^{\mathrm{g}} \mapsto [X_{i;t}]$ is obtained by \Cref{thm: main theorem for patterns}. The second isomorphism $\mathbb{Z}\tilde{\mathbb{S}}[[X_{i;t}] \mid i,t] \cong \mathcal{A}^{\mathrm{c}}/(\mathcal{I})$, $[X_{i;t}] \mapsto [X_{i;t}]$ is due to the standard isomorphism theorem. The map $\Psi: \mathcal{A}^{\mathrm{g}} \to \mathcal{A}^{\mathrm{c}}/\mathcal{I}$ is the composition of them.
\end{proof}

\section{Relations for $C$-matrices, $G$-matrices, and $F$-polynomials}\label{sec: separation formulas}

In this section, we present the relations between
the $C$-matrices, the $G$-matrices, and the $F$-polynomials
of a generalized cluster pattern $\mathbf{\Sigma}^{\mathrm{g}}$
and those of the associated composite cluster pattern $\mathbf{\Sigma}^{\mathrm{c}}$. These relations are compatible with the correspondence \eqref{eq: isomorphism between the cluster pattern}.

\subsection{Separation formulas}

In parallel to the ordinary case \cite{FZ07},
to any generalized cluster pattern $\mathbf{\Sigma}^{\mathrm{g}}$ with $y$-variables, one can associate the collections of the \emph{$C$-matrices} $\{C^{\mathrm{g}}_{t}\}_{t\in \mathbb{T}_n}$, the \emph{$G$-matrices} $\{G^{\mathrm{g}}_{t}\}_{t\in \mathbb{T}_n}$, and the \emph{$F$-polynomials} $\{F^{\mathrm{g}}_{i;t}\}_{i=1,\dots,n; t\in \mathbb{T}_n}$  \cite{Nak15a}.
Here, each $F$-polynomial $F^{\mathrm{g}}_{i;t}(\mathbf{y},\mathbf{z})$ is a polynomial in formal variables $\mathbf{y}=(y_i)_{i=1}^n$ and $\mathbf{z}=(z_{il})_{(i,l) \in \mathcal{D}'}$ with (nonnegative) integer coefficients, where $\mathcal{D}'=\{(i,l) \mid i=1,\dots,n;\ l=1,\dots,r_i-1\}$.
Then, the $x$- and $y$-variables of $\mathbf{\Sigma}^{\mathrm{g}}$ are expressed by the following \emph{separation formulas} \cite{Nak15a, NR16}:
\begin{align}
x^{\mathrm{g}}_{i;t}&=
\left(
\prod_{j=1}^n (x^{\mathrm{g}}_{j})^{g^{\mathrm{g}}_{ji;t}}
\right)
\frac{
F^{\mathrm{g}}_{i;t}(\hat{\mathbf{y}}^{\mathrm{g}}, \hat{\mathbf{z}})
}
{F^{\mathrm{g}}_{i;t}\vert_{\mathbb{P}_0}(\mathbf{y}^{\mathrm{g}}, \mathbf{z})
}
,
\\
y^{\mathrm{g}}_{i;t}&=
\prod_{j=1}^n (y^{\mathrm{g}}_{j})^{c^{\mathrm{g}}_{ji;t}}
\prod_{j=1}^n
F^{\mathrm{g}}_{j;t}\vert_{\mathbb{P}_0}(\mathbf{y}^{\mathrm{g}}, \mathbf{z})^{b_{ji;t}},
\end{align}
where ${\bf x}^{\mathrm{g}}$, ${\bf y}^{\mathrm{g}}$, and $\hat{\bf y}^{\mathrm{g}}$
are the initial $x$-, $y$-, $\hat{y}$-variables of $\mathbf{\Sigma}^\mathrm{g}$,
and $\hat{\mathbf{z}}$ and  ${\mathbf{z}}$ are the coefficients of the initial exchange polynomials $Z_i(u)$
and their images of $\pi_{\mathbb{P}_0}$.
There is an unfortunate coincidence of notations $g$ and $\mathrm{g}$, and also, $c$ and $\mathrm{c}$
for entirely different meanings, and we ask the readers to carefully distinguish them.
\par
Similarly,
to the  composite cluster pattern $\mathbf{\Sigma}^{\mathrm{c}}$,
one can associate the collections of the \emph{$C$-matrices}
$\{C^{\mathrm{c}}_{t}\}_{t\in \mathbb{T}_n}$, the
\emph{$G$-matrices} $\{G^{\mathrm{c}}_{t}\}_{t\in \mathbb{T}_n}$,
and the \emph{$F$-polynomials} $\{F^{\mathrm{c}}_{il;t}\}_{(i,l)\in \mathcal{D}; t\in \mathbb{T}_n}$.
They are the subcollections of the  ones in \cite{FZ07}
 for the ordinary cluster pattern $\mathbf{\Sigma}$.
Here, each $F$-polynomial $F^{\mathrm{c}}_{il;t}(\tilde{\bf y})$ is a polynomial in formal variables $\tilde{\bf y}=(\tilde{y}_{il})_{(i,l)\in \mathcal{D}}$ with (nonnegative) integer coefficients.
The separation formulas for the $x$- and $y$-variables
of $\mathbf{\Sigma}^{\mathrm{c}}$ are given by
\begin{align}
x^{\mathrm{c}}_{il;t}&=
\biggl(
\prod_{(j,m)\in \mathcal{D}} (x^{\mathrm{c}}_{jm})^{g^{\mathrm{c}}_{jm,il;t}}
\biggr)
\frac{
F^{\mathrm{c}}_{il;t}(\hat{\mathbf{y}}^{\mathrm{c}})
}
{F^{\mathrm{c}}_{il;t}\vert_{\overline{\mathbb{P}}}(\mathbf{y}^{\mathrm{c}})
}
,\label{eq: separation formula for x}
\\
y^{\mathrm{c}}_{il;t}&=
\prod_{(j,m)\in \mathcal{D}} (y^{\mathrm{c}}_{jm})^{c^{\mathrm{c}}_{jm,il;t}}
\prod_{(j,m)\in \mathcal{D}}
F^{\mathrm{c}}_{jm;t}\vert_{\overline{\mathbb{P}}}(\mathbf{y}^{\mathrm{c}})^{b_{ji;t}},\label{eq: separation formula for y}
\end{align}
where $\mathbf{x}^{\mathrm{c}}$, $\mathbf{y}^{\mathrm{c}}$, and $\hat{\mathbf{y}}^{\mathrm{c}}$
are  the initial $x$-, $y$-, $\hat{y}$-variables of $\mathbf{\Sigma}^{\mathrm{c}}$.

\subsection{Relations for $C$-matrices and $G$-matrices}
For any $i,j = 1,\dots,n$, set $\delta_{ij}=1$ if $i=j$ and $\delta_{ij}=0$ if $i \neq j$.
The $C$-matrices and the $G$-matrices are determined by the initial condition $c_{ij;t_0}^{\mathrm{g}}=g_{ij;t_0}^{\mathrm{g}}=\delta_{ij}$ and the following recursions for each pair of $k$-adjacent vertices $t,t' \in \mathbb{T}_n$ \cite{Nak15a, NR16}:
\begin{align}
c_{ij;t'}^{\mathrm{g}}&=\begin{cases}
-c_{ik;t}^{\mathrm{g}} & j=k,\\
c_{ij;t}^{\mathrm{g}}+r_k(c_{ik;t}^{\mathrm{g}}[b_{kj;t}]_{+}+[-c_{ik;t}^{\mathrm{g}}]_{+}b_{kj;t}) & j \neq k,
\end{cases}\label{eq: c-mutation}
\\
g_{ij;t'}^{\mathrm{g}}&=\begin{cases}
\displaystyle{-g_{ik;t}^{\mathrm{g}}+r_k\sum_{a=1}^{n}\left(g_{ia;t}^{\mathrm{g}}[-b_{ak;t}]_{+}-b_{ia;t_0}[-c_{ak;t}^{\mathrm{g}}]_{+}\right)} & j=k,\\
g_{ij;t}^{\mathrm{g}} & j \neq k.
\end{cases}\label{eq: g-mutation}
\end{align}
\par
For any $(i,l),(j,m) \in \mathcal{D}$, let $\delta_{il,jm}=1$ if $(i,l)=(j,m)$ and $\delta_{il,jm}=0$ if $(i,l) \neq (j,m)$. 
Recall that the sign $\sigma_{i;t}$ is the one in \eqref{eq: definition of sigma}. Then, the following fact holds.
\begin{lemma}\label{lem: c g lemma}
Let $\tilde{c}_{il,jm;t}^{\mathrm{c}}=c_{il,jm;t}^{\mathrm{c}}-\sigma_{j;t}\delta_{il,jm}$ and $\tilde{g}_{il,jm;t}^{\mathrm{c}}=g_{il,jm;t}^{\mathrm{c}}-\sigma_{j;t}\delta_{il,jm}$. Then, they are independent of $l=1,\dots,r_i$ and $m=1,\dots,r_j$.
\end{lemma}
\begin{proof}
For $t=t_0$, since $\tilde{c}_{il,jm;t_0}^{\mathrm{c}}=\tilde{g}_{il,jm;t_0}^{\mathrm{c}}=0$, the claim holds. Suppose that the claim holds for some $t \in \mathbb{T}_n$. Let $t' \in \mathbb{T}_n$ be $k$-adjacent to $t$. Consider the $C$-matrix $C_{t'}^{\mathrm{c}}$. By applying \eqref{eq: c-mutation} with ${\bf r}=(1,\dots,1)$ repeatedly, we have
\begin{equation}\label{eq: c composite mutation}
c_{il,jm;t'}^{\mathrm{c}}=\begin{cases}
-c_{il,km;t}^{\mathrm{c}} & j=k,\\
\displaystyle{c_{il,jm;t}^{\mathrm{c}}+\sum_{p=1}^{r_k} (c_{il,kp;t}^{\mathrm{c}}[b_{kj;t}]_{+}+[-c_{il,kp;t}^{\mathrm{c}}]_{+}b_{kj;t})} & j \neq k.
\end{cases}
\end{equation}
When $j=k$, the claim holds because $\tilde{c}_{il,km;t'}^{\mathrm{c}}=-\tilde{c}_{il,km;t}^{\mathrm{c}}$. Consider the case $j \neq k$. By the assumption, the sums $\sum_{p=1}^{r_k}c_{il,kp;t}^{\mathrm{c}}$ and $\sum_{p=1}^{r_k}[-c_{il,kp;t}^{\mathrm{c}}]_{+}$ are independent of $l=1,\dots,n$. Therefore, we have the claim by \eqref{eq: c composite mutation}. For the $G$-matrix $G_{t'}^{\mathrm{c}}$, when $j \neq k$, we have $\tilde{g}_{il,jm;t'}^{\mathrm{c}}=\tilde{g}_{il,jm;t}^{\mathrm{c}}$. Thus, the claim holds. When $j=k$, we have
\begin{equation}\label{eq: g composite mutation}
\begin{aligned}
g_{il,km;t'}^{\mathrm{c}}
=
-g_{il,km;t}^{\mathrm{c}}+\sum_{a=1}^{n}\sum_{p=1}^{r_a}\left(g_{il,ap;t}^{\mathrm{c}}[-b_{ak;t}]_{+}-b_{ia;t_0}[-c_{ap,km;t}^{\mathrm{c}}]_{+}\right).
\end{aligned}
\end{equation}
Then, the rest of the argument is similar to the above.
\end{proof}
In view of \Cref{lem: c g lemma}, we set
\begin{equation}\label{eq: definition of tilde c g}
\tilde{c}_{ij;t}^{\mathrm{c}}=c_{il,jm;t}^{\mathrm{c}}-\sigma_{j;t}\delta_{il,jm},
\quad
\tilde{g}_{ij;t}^{\mathrm{c}}=g_{il,jm;t}^{\mathrm{c}}-\sigma_{j;t}\delta_{il,jm},
\end{equation}
which are independent of $l$ and $m$.
Then, we have the following relations.
\begin{proposition}\label{prop: relationship in C-matrices and G-matrices}
The following equalities hold:
\begin{align}
c_{ij;t}^{\mathrm{g}}&=\sum_{l=1}^{r_i}c_{il,jm_0;t}^{\mathrm{c}}=\frac{r_i}{r_j}\sum_{m=1}^{r_j}c_{il_0,jm;t}^{\mathrm{c}}=r_i\tilde{c}_{ij;t}^{\mathrm{c}}+\sigma_{j;t}\delta_{ij}, \label{eq: C-relation}
\\
g_{ij;t}^{\mathrm{g}}&=\sum_{m=1}^{r_j}g_{il_0,jm;t}^{\mathrm{c}}=\frac{r_j}{r_i}\sum_{l=1}^{r_i}g_{il,jm_0;t}^{\mathrm{c}}=r_j\tilde{g}_{ij;t}^{\mathrm{c}}+\sigma_{j;t}\delta_{ij}, \label{eq: G-relation}
\end{align}
where $l_0=1,\dots,r_i$ and $m_0=1,\dots,r_j$ are arbitrary.
\end{proposition}
\begin{proof}
For $t=t_0$, the claim holds by definition. Suppose that the claim holds for some $t \in \mathbb{T}_n$, and let $t' \in \mathbb{T}_n$ be $k$-adjacent to $t$. Consider the $C$-matrix $C_{t'}^{\mathrm{c}}$. By \eqref{eq: c composite mutation}, the claim holds when $j=k$. When $j \neq k$, we have
\begin{equation}
\sum_{l=1}^{r_i}c_{il,jm;t'}^{\mathrm{c}}=\sum_{l=1}^{r_i}c_{il,jm;t}^{\mathrm{c}}+\sum_{p=1}^{r_k}\left(\sum_{l=1}^{r_i}c_{il,kp;t}^{\mathrm{c}}[b_{kj;t}]_{+}+\sum_{l=1}^{r_i}[-c_{il,kp;t}^{\mathrm{c}}]_{+}b_{kj;t}\right).
\end{equation}
Since $\sigma_{j;t} \in \{\pm 1\}$, $c_{il,kp;t}^{\mathrm{c}}$ ($l=1,\dots,r_i$) are either all nonnegative or all nonpositive. Thus, we have
\begin{equation}\label{eq: c commutativity}
\sum_{l=1}^{r_i}[-c_{il,kp;t}^{\mathrm{c}}]_{+}=[-c_{ik;t}^{\mathrm{g}}]_{+}.
\end{equation}
Then, we have
\begin{equation}
\begin{aligned}
\sum_{l=1}^{r_i}c_{il,jm;t'}^{\mathrm{c}}
=
c_{ij;t}^{\mathrm{g}}+r_k(c_{ik;t}^{\mathrm{g}}[b_{kj;t}]_{+}+[-c_{ik;t}^{\mathrm{g}}]_{+}b_{kj;t})\overset{\eqref{eq: c-mutation}}{=}c_{ij;t'}^{\mathrm{g}}.
\end{aligned}
\end{equation}
Thus, the first equality in \eqref{eq: C-relation} holds. Then, by \eqref{eq: definition of tilde c g}, we also obtain the others in \eqref{eq: C-relation}. For the $G$-matrix $G_{t'}^{\mathrm{c}}$, when $j \neq k$, the claim holds because $g_{il,jm;t'}^{\mathrm{c}}=g_{il,jm;t}^{\mathrm{c}}$. When $j=k$, by \eqref{eq: g composite mutation} and \eqref{eq: c commutativity}, we have
\begin{equation}
\begin{aligned}
\sum_{m=1}^{r_j}g_{il,km;t'}^{\mathrm{c}}=-g_{ik;t}^{\mathrm{g}}+r_j\sum_{a=1}^{n}\left(g_{ia;t}^{\mathrm{g}}[-b_{ak;t}]_{+}-b_{ia;t_0}[-c_{ak;t}^{\mathrm{g}}]_{+}\right)\overset{\eqref{eq: g-mutation}}{=}g_{ik;t'}^{\mathrm{g}}.
\end{aligned}
\end{equation}
Thus, the first equality in \eqref{eq: G-relation} holds. Then, by \eqref{eq: definition of tilde c g}, we also obtain the others in \eqref{eq: G-relation}.
\end{proof}

\subsection{Relation for $F$-polynomials}
For each $i=1,\dots,n$, let
\begin{equation}
\mathcal{Z}_{i;t_0}(u)=1+z_{i1}u+z_{i2}u^2+\cdots+z_{ir_i-1}u^{r_i-1}+u^{r_i},
\end{equation}
and we recursively define $\mathcal{Z}_{i;t}(u)$ ($t \in \mathbb{T}_n$) by $\mathcal{Z}_{i;t'}(u)=\mathcal{Z}_{i;t}(u)$ if $i \neq k$ and $\mathcal{Z}_{k;t'}(u)=\mathcal{Z}_{k;t}^{\mathrm{rec}}(u)$ for any pair of $k$-adjacent vertices $t,t' \in \mathbb{T}_n$.
The $F$-polynomials are determined by the initial condition $F_{i;t_0}=1$ ($i=1,\dots,n$) and the following recursion for each pair of $k$-adjacent vertices $t,t' \in \mathbb{T}_n$ \cite{Nak15a, NR16}:
\begin{align}
F_{i;t'}^{\mathrm{g}}({\bf y},{\bf z})&=\begin{cases}
F_{k;t}^{\mathrm{g}}({\bf y},{\bf z})^{-1} \cdot M({\bf y},{\bf z}) & i=k,\\
F_{i;t}^{\mathrm{g}}({\bf y},{\bf z}) & i \neq k,
\end{cases}\label{eq: F-mutation}
\end{align}
where
\begin{equation}\label{eq: M polynomial}
M({\bf y},{\bf z})=\left(\prod_{j=1}^{n}y_j^{[-c_{jk;t}^{\mathrm{g}}]_{+}}F_{j;t}^{\mathrm{g}}({\bf y},{\bf z})^{[-b_{jk;t}]_{+}}\right)^{r_k}\mathcal{Z}_{k;t}\left(\prod_{j=1}^{n}y_j^{c_{jk;t}^{\mathrm{g}}}F_{j;t}^{\mathrm{g}}({\bf y},{\bf z})^{b_{jk;t}}\right).
\end{equation}
\par
Recall that each $F$-polynomial of the composite cluster pattern ${\bf \Sigma}^{\mathrm{c}}$ is a polynomial in $\tilde{\bf y}=(\tilde{y}_{il})_{(i,l) \in \mathcal{D}}$.
Let $\mathfrak{S}(\mathcal{D})$ be the symmetric group of $\mathcal{D}$. For any $\sigma \in \mathfrak{S}(\mathcal{D})$, we write
\begin{equation}
\tilde{\bf y}^{\sigma}=(\tilde{y}_{\sigma(i,l)})_{(i,l) \in \mathcal{D}}.
\end{equation}
For each $j=1,\dots,n$, let $\mathcal{D}_j=\{(j,m) \in \mathcal{D} \mid m=1,\dots,r_j\}$. Then, the following fact holds.
\begin{lemma}\label{lem: symmrtry of F-polynomials}
Let $(i,l) \in \mathcal{D}$ and $t \in \mathbb{T}_n$.
Let $\sigma \in \mathfrak{S}(\mathcal{D})$ satisfy $\sigma(\mathcal{D}_j) = \mathcal{D}_j$ for any $j=1,\dots,n$. Then, we have
\begin{equation}\label{eq: symmetry of F-polynomials}
F_{il;t}^{\mathrm{c}}(\tilde{\bf y}^{\sigma})=F_{\sigma(i,l);t}^{\mathrm{c}}(\tilde{\bf y}).
\end{equation}
\end{lemma}
\begin{proof}
Suppose that the claim holds for some $t \in \mathbb{T}_n$.
Let $t' \in \mathbb{T}_n$ be $k$-adjacent to $t$.
Set
\begin{equation}\label{eq: polynomial P Q}
\begin{aligned}
P_{kl}(\tilde{\bf y})&=\prod_{(j,m) \in \mathcal{D}} \tilde{y}_{jm}^{[-c_{jm,kl;t}^{\mathrm{c}}]_{+}}F_{jm;t}^{\mathrm{c}}(\tilde{\bf y})^{[-b_{jk;t}]_{+}},
\\
Q_{kl}(\tilde{\bf y})&=1+\prod_{(j,m) \in \mathcal{D}} \tilde{y}_{jm}^{c_{jm,kl;t}^{\mathrm{c}}}F_{jm;t}^{\mathrm{c}}(\tilde{\bf y})^{b_{jk;t}}.
\end{aligned}
\end{equation}
By \eqref{eq: F-mutation} with ${\bf r}=(1,\dots,1)$, we have $F_{kl;t'}^{\mathrm{c}}(\tilde{\bf y})=F_{kl;t}^{\mathrm{c}}(\tilde{\bf y})^{-1}P_{kl}(\tilde{\bf y})Q_{kl}(\tilde{\bf y})$.
Also, by \Cref{lem: c g lemma}, $c_{\sigma^{-1}(j,m),kl}^{\mathrm{c}}=c_{jm,\sigma(k,l)}^{\mathrm{c}}$ holds. So, we have $P_{kl}(\tilde{\bf y}^{\sigma})=P_{\sigma(k,l)}(\tilde{\bf y})$ and $Q_{kl}(\tilde{\bf y}^{\sigma})=Q_{\sigma(k,l)}(\tilde{\bf y})$. Thus, the claim holds.
\end{proof}
We introduce formal variables $s_{jm}$ ($(j,m) \in \mathcal{D}$), and we set
\begin{equation}
{\bf s}=(s_{jm})_{(j,m) \in \mathcal{D}},
\quad
{\bf s}{\bf y}=(s_{jm}y_j)_{(j,m) \in \mathcal{D}}.
\end{equation}
Then, by \Cref{lem: symmrtry of F-polynomials}, for each $i$ and $j$, the product $\prod_{l=1}^{r_i}F_{il;t}^{\mathrm{c}}({\bf sy})$ is symmetric with respect to $s_{j1},\dots,s_{jr_j}$. In particular, it can be expressed as a polynomial in ${\bf y}$ and ${\bf e}=(e_{jm})_{(j,m) \in \mathcal{D}}$, where $e_{jm}$ is the elementary symmetric polynomial in $s_{j1},\dots,s_{jr_j}$ of degree $m$. In this expression, we apply the substitution $e_{jm} = z_{jm}$ ($(j,m) \in \mathcal{D}'$) and $e_{jr_j} = 1$ ($j=1,\dots,n$). The resulting polynomial in ${\bf y}$ and ${\bf z}$ is denoted by $\prod_{l=1}^{r_i}F_{il;t}^{\mathrm{c}}({\bf sy})|_{{\bf e}={\bf z}}$, for simplicity, where we slightly abuse the symbol ${\bf z}$.
\par
Let $\tau_{il,im} \in \mathfrak{S}(\mathcal{D})$ be the transposition of $(i,l)$ and $(i,m)$ for $l\neq m$ and the identity for $l=m$.
Then, we have the following relations.
\begin{proposition}\label{prop: F-relation}
For arbitrary $l_0=1,\dots,r_i$, the following equalities hold:
\begin{equation}\label{eq: F-relation}
F_{i;t}^{\mathrm{g}}({\bf y},{\bf z})=\left.\prod_{l=1}^{r_i}F_{il;t}^{\mathrm{c}}({\bf s}{\bf y})\right|_{{\bf e} = {\bf z}}=\left.\prod_{m=1}^{r_i}F_{il_0;t}^{\mathrm{c}}\left(({\bf s}{\bf y})^{\tau_{il_0,im}}\right)\right|_{{\bf e} = {\bf z}}.
\end{equation}
\end{proposition}
\begin{proof}
The second equality follows from \Cref{lem: symmrtry of F-polynomials}. We show the first one. For $t=t_0$, the claim holds by definition.
Suppose that the claim holds for some $t \in \mathbb{T}_n$, and let $t' \in \mathbb{T}_n$ be $k$-adjacent to $t$. Then, we have $F_{k;t'}^{\mathrm{g}}({\bf y},{\bf z})=F_{k;t}^{\mathrm{g}}({\bf y},{\bf z})^{-1}M({\bf y},{\bf z})$, where $M$ is the polynomial defined in \eqref{eq: M polynomial}. On the other hand, we have $\prod_{l=1}^{r_k}F_{kl;t'}^{\mathrm{c}}({\bf s}{\bf y})=(\prod_{l=1}^{r_k}F_{kl;t}^{\mathrm{c}}({\bf s}{\bf y}))^{-1}P({\bf s}{\bf y})Q({\bf s}{\bf y})$, where
\begin{align}
P({\bf sy})&=\prod_{l=1}^{r_k}\prod_{(j,m) \in \mathcal{D}}(s_{jm}y_j)^{[-c_{jm,kl;t}^{\mathrm{c}}]_{+}}F_{jm;t}^{\mathrm{c}}({\bf sy})^{[-b_{jk;t}]_{+}},
\\
Q({\bf sy})&=\prod_{l=1}^{r_k}\biggl(1+\prod_{(j,m) \in \mathcal{D}}(s_{jm}y_{j})^{c_{jm,kl;t}^{\mathrm{c}}}F_{jm;t}^{\mathrm{c}}({\bf sy})^{b_{jk;t}}\biggr).\label{eq: polynomial Q}
\end{align}
Comparing the first and the third terms in \eqref{eq: C-relation}, we have
\begin{equation}
\prod_{l=1}^{r_k}\prod_{(j,m) \in \mathcal{D}}s_{jm}^{[-c_{jm,kl;t}^{\mathrm{c}}]_{+}}=\prod_{(j,m) \in \mathcal{D}}s_{jm}^{\frac{r_k}{r_j}[-c_{jk;t}^{\mathrm{g}}]_{+}}=\prod_{j=1}^{n}e_{jr_j}^{\frac{r_k}{r_j}[-c_{jk;t}^{\mathrm{c}}]_{+}}.
\end{equation}
Thus, by \eqref{eq: C-relation}, we obtain
\begin{equation}\label{eq: reduced form of the first part}
\begin{aligned}
P({\bf s}{\bf y})|_{{\bf e}={\bf z}} = 
\prod_{j=1}^{n}y_j^{r_k[-c_{jk;t}^{\mathrm{g}}]_{+}}F_{j;t}^{\mathrm{g}}({\bf y},{\bf z})^{r_k[-b_{jk;t}]_{+}}.
\end{aligned}
\end{equation}
Meanwhile, by \eqref{eq: definition of tilde c g}, we have
\begin{equation}\label{eq: s to one factor}
\prod_{(j,m) \in \mathcal{D}}s_{jm}^{c_{jm,kl;t}^{\mathrm{c}}}=s_{kl}^{\sigma_{k;t}}\prod_{j=1}^{n}e_{jr_j}^{\tilde{c}_{jk;t}^{\mathrm{c}}}.
\end{equation}
Thus, we have
\begin{equation}\label{eq: reduced form of Q}
\begin{aligned}
Q({\bf sy})|_{{\bf e}={\bf z}}
\overset{\eqref{eq: C-relation}}&{=}
\left.\prod_{l=1}^{r_k}\left(1+s_{kl}^{\sigma_{k;t}}\prod_{j=1}^{n}y_j^{c_{jk;t}^{\mathrm{g}}}F_{j;t}^{\mathrm{g}}({\bf y},{\bf z})^{b_{jk;t}}\right)\right|_{{\bf e}={\bf z}}
\\
&=\mathcal{Z}_{k;t}\left(\prod_{j=1}^{n}y_j^{c_{jk;t}^{\mathrm{g}}}F_{j;t}^{\mathrm{g}}({\bf y},{\bf z})^{b_{jk;t}}\right),
\end{aligned}
\end{equation}
where the second equality is shown by the same way as \eqref{eq: equivalence in P}.
By \eqref{eq: reduced form of the first part} and \eqref{eq: reduced form of Q}, we have $P({\bf s}{\bf y})Q({\bf s}{\bf y})|_{{\bf e}={\bf z}} = M({\bf y},{\bf z})$. Thus, the claim holds.
\end{proof}

\begin{remark}
The equality \eqref{eq: F-relation} does {\em not} directly imply the positivity of the coefficients in $F_{i;t}^{\mathrm{g}}$ (the Laurent positivity of $x_{i;t}^{\mathrm{g}}$ proved by \cite{BLM25b}) from the one of $F_{il;t}^{\mathrm{c}}$ due to possible signs such as  $s_1^2+s_2^2=(s_1+s_2)^2-2s_1s_2$.
\end{remark}

\subsection{Examples}
Let us give some examples of Propositions~\ref{prop: relationship in C-matrices and G-matrices} and \ref{prop: F-relation}.
We consider generalized cluster patterns of rank $2$. Let $t_0$ be the initial vertex, and consider the following vertices:
\begin{equation}
\begin{tikzpicture}
\draw (0,0)--(6,0);
\foreach \x in {0,1,2,3}
    {
    \fill ({2*\x},0) circle [radius=0.06];
    \draw ({2*\x},0) node [above] {$t_{\x}$};
    }
\draw (1,0) node [above] {$1$};
\draw (3,0) node [above] {$2$};
\draw (5,0) node [above] {$1$};
\end{tikzpicture}
\end{equation}
\par
(1) We consider the case
\begin{equation}\label{eq: conditions for the 1st example}
\begin{gathered}
B_{t_0}=\left(\begin{matrix}
0 & -1\\
1 & 0
\end{matrix}\right),
\ 
Z_{1;t_0}(u)=1+\hat{z}u+u^2,
\ 
Z_{2;t_0}(u)=1+u,
\\
\mathcal{B}_{t_0}=\left(\begin{array}{cc|c}
0 & 0 & -1\\
0 & 0 & -1\\
\hline
1 & 1 & 0
\end{array}\right).
\end{gathered}
\end{equation}
The $C$-matrices, the $G$-matrices, and the $F$-polynomials are given in \Cref{tab: 1st example}.
\begin{table}[htbp]
\centering
{
\begin{tabular}{c||c|c||c|c}
$t$ & $C^{\mathrm{g}}_t$ & $C^{\mathrm{c}}_t$ & $G_{t}^{\mathrm{g}}$ & $G_{t}^{\mathrm{c}}$\\
\hline
\rule[-2.5ex]{0mm}{6.25ex}
$t_0$
&
${\tiny
\renewcommand{\arraystretch}{1}
\renewcommand{\arraycolsep}{1pt}
\left(\begin{array}{cc}
1 & 0 \\
0 & 1 
\end{array}\right)}$
&
${\tiny 
\renewcommand{\arraystretch}{1}
\renewcommand{\arraycolsep}{1pt} 
\left(\begin{array}{cc|c}
1 & 0 & 0\\
0 & 1 & 0\\
\hline
0 & 0 & 1
\end{array}\right)}$
&
${\tiny
\renewcommand{\arraystretch}{1}
\renewcommand{\arraycolsep}{1pt}
\left(\begin{array}{cc}
1 & 0\\
0 & 1
\end{array}\right)}$
&
${\tiny
\renewcommand{\arraystretch}{1}
\renewcommand{\arraycolsep}{1pt}
\left(\begin{array}{cc|c}
1 & 0 & 0\\
0 & 1 & 0\\
\hline
0 & 0 & 1
\end{array}\right)}$
\\
\hline
\rule[-2.5ex]{0mm}{6.25ex}
$t_1$
&
${\tiny
\renewcommand{\arraystretch}{1}
\renewcommand{\arraycolsep}{1pt}
\left(\begin{array}{cc}
-1 & 0\\
0 & 1
\end{array}\right)}$
&
${\tiny
\renewcommand{\arraystretch}{1}
\renewcommand{\arraycolsep}{1pt}
\left(\begin{array}{cc|c}
-1 & 0 & 0\\
0 & -1 & 0\\
\hline
0 & 0 & 1
\end{array}\right)}$
&
${\tiny
\renewcommand{\arraystretch}{1}
\renewcommand{\arraycolsep}{1pt}
\left(\begin{array}{cc}
-1 & 0\\
0 & 1
\end{array}\right)}$
&
${\tiny
\renewcommand{\arraystretch}{1}
\renewcommand{\arraycolsep}{1pt}
\left(\begin{array}{cc|c}
-1 & 0 & 0\\
0 & -1 & 0\\
\hline
0 & 0 & 1
\end{array}\right)}$
\\[2mm]
\hline
\rule[-2.5ex]{0mm}{6.25ex}
$t_2$
&
${\tiny
\renewcommand{\arraystretch}{1}
\renewcommand{\arraycolsep}{1pt}
\left(\begin{array}{cc}
-1 & 0\\
0 & -1
\end{array}\right)}$
&
${\tiny 
\renewcommand{\arraystretch}{1}
\renewcommand{\arraycolsep}{1pt}
\left(\begin{array}{cc|c}
-1 & 0 & 0\\
0 & -1 & 0\\
\hline
0 & 0 & -1
\end{array}\right)}$
&
${\tiny 
\renewcommand{\arraystretch}{1}
\renewcommand{\arraycolsep}{1pt}
\left(\begin{array}{cc}
-1 & 0\\
0 & -1
\end{array}\right)}$
&
${\tiny 
\renewcommand{\arraystretch}{1}
\renewcommand{\arraycolsep}{1pt}
\left(\begin{array}{cc|c}
-1 & 0 & 0\\
0 & -1 & 0\\
\hline
0 & 0 & -1
\end{array}\right)}$
\\
\hline
\rule[-2.5ex]{0mm}{6.25ex}
$t_3$
&
${\tiny 
\renewcommand{\arraystretch}{1}
\renewcommand{\arraycolsep}{1pt}
\left(\begin{array}{cc}
-1 & -2\\
0 & -1
\end{array}\right)}$
&
${\tiny 
\renewcommand{\arraystretch}{1}
\renewcommand{\arraycolsep}{1pt}
\left(\begin{array}{cc|c}
1 & 0 & -1\\
0 & 1 & -1\\
\hline
0 & 0 & -1
\end{array}\right)}$
&
${\tiny 
\renewcommand{\arraystretch}{1}
\renewcommand{\arraycolsep}{1pt}
\left(\begin{array}{cc}
1 & 0\\
-2 & -1
\end{array}\right)}$
&
${\tiny 
\renewcommand{\arraystretch}{1}
\renewcommand{\arraycolsep}{1pt}
\left(\begin{array}{cc|c}
1 & 0 & 0\\
0 & 1 & 0\\
\hline
-1 & -1 & -1
\end{array}\right)}$
\end{tabular}}
{\renewcommand{\arraystretch}{1.5}
\renewcommand{\arraycolsep}{1pt}
\begin{tabular}{c||c|c}
$t$ & $F^{\mathrm{g}}_{i;t}$ & $F^{\mathrm{c}}_{il;t}$\\
\hline
$t_0$
&
{\tiny
\renewcommand{\arraystretch}{0.9}
\renewcommand{\arraycolsep}{1pt}
$\begin{array}{c}
1\\
1
\end{array}$}
&
{\tiny
\renewcommand{\arraystretch}{0.9}
\renewcommand{\arraycolsep}{1pt}
$\begin{array}{c}
1\\
1\\
1
\end{array}$}
\\
\hline
$t_1$
&
{\tiny
\renewcommand{\arraystretch}{0.9}
\renewcommand{\arraycolsep}{1pt}
$\begin{array}{c}
1+zy_1+y_1^2\\
1
\end{array}$}
&
{\tiny
\renewcommand{\arraystretch}{0.9}
\renewcommand{\arraycolsep}{1pt}
$\begin{array}{c}
1+y_{11}\\
1+y_{12}\\
1
\end{array}$}
\\
\hline
$t_2$
&
{\tiny
\renewcommand{\arraystretch}{0.9}
\renewcommand{\arraycolsep}{1pt}
$\begin{array}{c}
1+zy_1+y_1^2\\
1+y_2+zy_1y_2+y_1^2y_2
\end{array}$}
&
{\tiny
\renewcommand{\arraystretch}{0.9}
\renewcommand{\arraycolsep}{1pt}
$\begin{array}{c}
1+y_{11}\\
1+y_{12}\\
1+y_{21}+(y_{11}+y_{12})y_{21}+y_{11}y_{12}y_{21}
\end{array}$}
\\
\hline
$t_3$
&
{\tiny
\renewcommand{\arraystretch}{0.9}
\renewcommand{\arraycolsep}{1pt}
$\begin{array}{c}
1+2y_2+y_2^2+zy_1y_2+zy_1y_2^2+y_1^2y_2^2\\
1+y_2+zy_1y_2+y_1^2y_2
\end{array}$}
&
{\tiny
\renewcommand{\arraystretch}{0.9}
\renewcommand{\arraycolsep}{1pt}
$\begin{array}{c}
1+y_{21}+y_{12}y_{21}\\
1+y_{21}+y_{11}y_{21}\\
1+y_{21}+(y_{11}+y_{12})y_{21}+y_{11}y_{12}y_{21}
\end{array}$}
\end{tabular}}
\caption{$C$- and $G$-matrices and $F$-polynomials for Case~(1).
In each box on the column of $F_{i;t}^{\mathrm{g}}$, we put $F_{1;t}^{\mathrm{g}}$ and $F_{2;t}^{\mathrm{g}}$, and, on the column of $F_{il;t}^{\mathrm{c}}$, we put $F_{11;t}^{\mathrm{c}}$, $F_{12;t}^{\mathrm{c}}$, and $F_{21;t}^{\mathrm{c}}$ in this order.}
\label{tab: 1st example}
\end{table}
\\
The equality \eqref{eq: F-relation} holds for $i=1$ and $t=t_3$ as follows:
\begin{equation}
\begin{aligned}
&\ F_{11;t_3}^{\mathrm{c}}(s_{11}y_1,s_{12}y_1,s_{21}y_2)F_{12;t_3}^{\mathrm{c}}(s_{11}y_1,s_{12}y_1,s_{21}y_2)|_{{\bf e}={\bf z}}
\\
=&\ 
\{1+2s_{21}y_2+s_{21}^2y_2^2+(s_{11}+s_{12})s_{21}y_1y_2
\\
&\qquad
+(s_{11}+s_{12})s_{21}^2y_1y_2^2+s_{11}s_{12}s_{21}^2y_1^2y_2^2\}|_{{\bf e}={\bf z}}
\\
=&\ 
F_{1;t_3}^{\mathrm{g}}(y_1,y_2,z).
\end{aligned}
\end{equation}
\par
(2) To illustrate \Cref{prop: relationship in C-matrices and G-matrices} more evidently, we consider the case where the matrix $\mathcal{B}_{t_0}$ is the one in \eqref{eq: example of an enlargement}. The $C$-matrices and the $G$-matrices are given in \Cref{tab: 2nd example}.
\begin{table}[htbp]
{
\begin{tabular}{c||c|c||c|c}
$t$ & $C_t^{\mathrm{g}}$ & $C_t^{\mathrm{c}}$ & $G_t^{\mathrm{g}}$ & $G_t^{\mathrm{c}}$
\\
\hline
\rule[-4ex]{0mm}{9.25ex}
$t_0$
&
${\tiny
\renewcommand{\arraystretch}{1}
\renewcommand{\arraycolsep}{1pt}
\left(\begin{array}{cc}
1 & 0\\
0 & 1
\end{array}\right)}$
&
${\tiny 
\renewcommand{\arraystretch}{1}
\renewcommand{\arraycolsep}{1pt}
\left(\begin{array}{cc|ccc}
1 & 0 & 0 & 0 & 0\\
0 & 1 & 0 & 0 & 0\\
\hline
0 & 0 & 1 & 0 & 0\\
0 & 0 & 0 & 1 & 0\\
0 & 0 & 0 & 0 & 1
\end{array}\right)}$
&
${\tiny 
\renewcommand{\arraystretch}{1}
\renewcommand{\arraycolsep}{1pt}
\left(\begin{array}{cc}
1 & 0\\
0 & 1
\end{array}\right)}$
&
${\tiny
\renewcommand{\arraystretch}{1}
\renewcommand{\arraycolsep}{1pt}
\left(\begin{array}{cc|ccc}
1 & 0 & 0 & 0 & 0\\
0 & 1 & 0 & 0 & 0\\
\hline
0 & 0 & 1 & 0 & 0\\
0 & 0 & 0 & 1 & 0\\
0 & 0 & 0 & 0 & 1
\end{array}\right)}$
\\
\hline
\rule[-4ex]{0mm}{9.25ex}
$t_1$
&
${\tiny
\renewcommand{\arraystretch}{1}
\renewcommand{\arraycolsep}{1pt}
\left(\begin{array}{cc}
-1 & 2\\
0 & 1
\end{array}\right)}$
&
${\tiny
\renewcommand{\arraystretch}{1}
\renewcommand{\arraycolsep}{1pt}
\left(\begin{array}{cc|ccc}
-1 & 0 & 1 & 1 & 1\\
0 & -1 & 1 & 1 & 1\\
\hline
0 & 0 & 1 & 0 & 0\\
0 & 0 & 0 & 1 & 0\\
0 & 0 & 0 & 0 & 1
\end{array}\right)}$
&
${\tiny
\renewcommand{\arraystretch}{1}
\renewcommand{\arraycolsep}{1pt}
\left(\begin{array}{cc}
-1 & 0\\
4 & 1
\end{array}\right)}$
&
${\tiny
\renewcommand{\arraystretch}{1}
\renewcommand{\arraycolsep}{1pt}
\left(\begin{array}{cc|ccc}
-1 & 0 & 0 & 0 & 0\\
0 & -1 & 0 & 0 & 0\\
\hline
2 & 2 & 1 & 0 & 0\\
2 & 2 & 0 & 1 & 0\\
2 & 2 & 0 & 0 & 1
\end{array}\right)}$
\\
\hline
\rule[-4ex]{0mm}{9.25ex}
$t_2$
&
${\tiny
\renewcommand{\arraystretch}{1}
\renewcommand{\arraycolsep}{1pt}
\left(\begin{array}{cc}
11 & -2\\
6 & -1
\end{array}\right)}$
&
${\tiny
\renewcommand{\arraystretch}{1}
\renewcommand{\arraycolsep}{1pt}
\left(\begin{array}{cc|ccc}
5 & 6 & -1 & -1 & -1\\
6 & 5 & -1 & -1 & -1\\
\hline
2 & 2 & -1 & 0 & 0\\
2 & 2 & 0 & -1 & 0\\
2 & 2 & 0 & 0 & -1
\end{array}\right)}$
&
${\tiny
\renewcommand{\arraystretch}{1}
\renewcommand{\arraycolsep}{1pt}
\left(\begin{array}{cc}
-1 & -3\\
4 & 11
\end{array}\right)}$
&
${\tiny
\renewcommand{\arraystretch}{1}
\renewcommand{\arraycolsep}{1pt}
\left(\begin{array}{cc|ccc}
-1 & 0 & -1 & -1 & -1\\
0 & -1 & -1 & -1 & -1\\
\hline
2 & 2 & 3 & 4 & 4\\
2 & 2 & 4 & 3 & 4\\
2 & 2 & 4 & 4 & 3
\end{array}\right)}$
\\
\hline
\rule[-4ex]{0mm}{9.25ex}
$t_3$
&
${\tiny
\renewcommand{\arraystretch}{1}
\renewcommand{\arraycolsep}{1pt}
\left(\begin{array}{cc}
-11 & 20\\
-6 & 11
\end{array}\right)}$
&
${\tiny
\renewcommand{\arraystretch}{1}
\renewcommand{\arraycolsep}{1pt}
\left(\begin{array}{cc|ccc}
-5 & -6 & 10 & 10 & 10\\
-6 & -5 & 10 & 10 & 10\\
\hline
-2 & -2 & 3 & 4 & 4\\
-2 & -2 & 4 & 3 & 4\\
-2 & -2 & 4 & 4 & 3
\end{array}\right)}$
&
${\tiny
\renewcommand{\arraystretch}{1}
\renewcommand{\arraycolsep}{1pt}
\left(\begin{array}{cc}
-11 & -3\\
40 & 11
\end{array}\right)}$
&
${\tiny
\renewcommand{\arraystretch}{1}
\renewcommand{\arraycolsep}{1pt}
\left(\begin{array}{cc|ccc}
-5 & -6 & -1 & -1 & -1\\
-6 & -5 & -1 & -1 & -1\\
\hline
20 & 20 & 3 & 4 & 4\\
20 & 20 & 4 & 3 & 4\\
20 & 20 & 4 & 4 & 3
\end{array}\right)}$
\end{tabular}}
\caption{$C$- and $G$-matrices for Case~(2).}
\label{tab: 2nd example}
\end{table}

\FloatBarrier
\bibliography{Relation_between_generalized_and_ordinary_cluster_algebras}

@preamble{"\newcommand{\noop}[1]{}"}

@incollection {NR16,
    AUTHOR = {Nakanishi, Tomoki and Rupel, Dylan},
     TITLE = {Companion cluster algebras to a generalized cluster algebra},
 BOOKTITLE = {Travaux math\'ematiques. {V}ol. {XXIV}},
    SERIES = {Trav. Math.},
    VOLUME = {24},
     PAGES = {129--149},
 PUBLISHER = {Fac. Sci. Technol. Commun. Univ. Luxemb., Luxembourg},
      YEAR = {2016},
      ISBN = {978-2-87971-167-6},
   MRCLASS = {13F60},
  MRNUMBER = {3643935},
MRREVIEWER = {Fan\ Qin},
}

@misc{RW25,
  title={Generalized cluster algebras are subquotients of cluster algebras},
  author={Ramos, Rolando and Whiting, David},
  howpublished={arXiv preprint arXiv:2504.20931},
  year={2025}
}

@article {FZ02,
    AUTHOR = {Fomin, Sergey and Zelevinsky, Andrei},
     TITLE = {Cluster algebras. {I}. {F}oundations},
   JOURNAL = {J. Amer. Math. Soc.},
  FJOURNAL = {Journal of the American Mathematical Society},
    VOLUME = {15},
      YEAR = {2002},
    NUMBER = {2},
     PAGES = {497--529},
      ISSN = {0894-0347,1088-6834},
   MRCLASS = {16S99 (14M99 17B99)},
  MRNUMBER = {1887642},
MRREVIEWER = {Eric\ N.\ Sommers},
       DOI = {10.1090/S0894-0347-01-00385-X},
       URL = {https://doi.org/10.1090/S0894-0347-01-00385-X},
}

@article {Nak15a,
    AUTHOR = {Nakanishi, Tomoki},
     TITLE = {Structure of seeds in generalized cluster algebras},
   JOURNAL = {Pacific J. Math.},
  FJOURNAL = {Pacific Journal of Mathematics},
    VOLUME = {277},
      YEAR = {\noop{a}2015},
    NUMBER = {1},
     PAGES = {201--217},
      ISSN = {0030-8730,1945-5844},
   MRCLASS = {13F60},
  MRNUMBER = {3393688},
MRREVIEWER = {Hugh\ Ross\ Thomas},
       DOI = {10.2140/pjm.2015.277.201},
       URL = {https://doi.org/10.2140/pjm.2015.277.201},
}

@article {CS14,
    AUTHOR = {Chekhov, Leonid and Shapiro, Michael},
     TITLE = {Teichm\"uller spaces of {R}iemann surfaces with orbifold
              points of arbitrary order and cluster variables},
   JOURNAL = {Int. Math. Res. Not. IMRN},
  FJOURNAL = {International Mathematics Research Notices. IMRN},
      YEAR = {2014},
    NUMBER = {10},
     PAGES = {2746--2772},
      ISSN = {1073-7928,1687-0247},
   MRCLASS = {32G15 (13F60 30F60)},
  MRNUMBER = {3214284},
       DOI = {10.1093/imrn/rnt016},
       URL = {https://doi.org/10.1093/imrn/rnt016},
}

@misc{Nak24,
  title={Addendum to {``Structure of seeds in generalized cluster algebras"}},
  author={Nakanishi, Tomoki},
  howpublished={arXiv preprint arXiv:2406.07582},
  year={2024}
}

@article {FZ07,
    AUTHOR = {Fomin, Sergey and Zelevinsky, Andrei},
     TITLE = {Cluster algebras. {IV}. {C}oefficients},
   JOURNAL = {Compos. Math.},
  FJOURNAL = {Compositio Mathematica},
    VOLUME = {143},
      YEAR = {2007},
    NUMBER = {1},
     PAGES = {112--164},
      ISSN = {0010-437X,1570-5846},
   MRCLASS = {16S99 (05E15 14M17 22E46)},
  MRNUMBER = {2295199},
MRREVIEWER = {Christof\ Gei\ss},
       DOI = {10.1112/S0010437X06002521},
       URL = {https://doi.org/10.1112/S0010437X06002521},
}

@article {Nak15b,
    AUTHOR = {Nakanishi, Tomoki},
     TITLE = {Quantum generalized cluster algebras and quantum dilogarithm
              functions of higher degrees},
   JOURNAL = {Teoret. Mat. Fiz.},
  FJOURNAL = {Teoreticheskaya i Matematicheskaya Fizika},
    VOLUME = {185},
      YEAR = {\noop{b}2015},
    NUMBER = {3},
     PAGES = {460--470},
      ISSN = {0564-6162,2305-3135},
   MRCLASS = {81Q20 (13F60 33B30)},
  MRNUMBER = {3438630},
MRREVIEWER = {Rinat\ M.\ Kashaev},
       DOI = {10.4213/tmf8897},
       URL = {https://doi.org/10.4213/tmf8897},
}

@article {BCMX18,
    AUTHOR = {Bai, Liqian and Chen, Xueqing and Ding, Ming and Xu, Fan},
     TITLE = {A quantum analog of generalized cluster algebras},
   JOURNAL = {Algebr. Represent. Theory},
  FJOURNAL = {Algebras and Representation Theory},
    VOLUME = {21},
      YEAR = {2018},
    NUMBER = {6},
     PAGES = {1203--1217},
      ISSN = {1386-923X,1572-9079},
   MRCLASS = {13F60 (16G20 17B67 18E30)},
  MRNUMBER = {3874742},
MRREVIEWER = {Fang\ Li},
       DOI = {10.1007/s10468-017-9743-7},
       URL = {https://doi.org/10.1007/s10468-017-9743-7},
}

@article {FPY24,
    AUTHOR = {Fu, Changjian and Peng, Liangang and Ye, Huihui},
     TITLE = {On {$F$}-polynomials for generalized quantum cluster algebras
              and {G}upta's formula},
   JOURNAL = {SIGMA Symmetry Integrability Geom. Methods Appl.},
  FJOURNAL = {SIGMA. Symmetry, Integrability and Geometry. Methods and
              Applications},
    VOLUME = {20},
      YEAR = {2024},
     PAGES = {Paper No. 080, 26},
      ISSN = {1815-0659},
   MRCLASS = {13F60 (05E16 16S34)},
  MRNUMBER = {4843369},
MRREVIEWER = {Emine\ Y\i ld\i r\i m},
       DOI = {10.3842/SIGMA.2024.080},
       URL = {https://doi.org/10.3842/SIGMA.2024.080},
}

@article {CKM21,
    AUTHOR = {Cheung, Man-Wai and Kelley, Elizabeth and Musiker, Gregg},
     TITLE = {Cluster scattering diagrams and theta basis for reciprocal
              generalized cluster algebras},
   JOURNAL = {S\'em. Lothar. Combin.},
  FJOURNAL = {S\'eminaire Lotharingien de Combinatoire},
    VOLUME = {85B},
      YEAR = {2021},
     PAGES = {Art. 86, 12},
      ISSN = {1286-4889},
   MRCLASS = {13F60},
  MRNUMBER = {4311967},
}

@article {Mou24,
    AUTHOR = {Mou, Lang},
     TITLE = {Scattering diagrams for generalized cluster algebras},
   JOURNAL = {Algebra Number Theory},
  FJOURNAL = {Algebra \& Number Theory},
    VOLUME = {18},
      YEAR = {2024},
    NUMBER = {12},
     PAGES = {2179--2246},
      ISSN = {1937-0652,1944-7833},
   MRCLASS = {13F60},
  MRNUMBER = {4817459},
MRREVIEWER = {Liqian\ Bai},
       DOI = {10.2140/ant.2024.18.2179},
       URL = {https://doi.org/10.2140/ant.2024.18.2179},
}

@article {BLM25a,
    AUTHOR = {Burcroff, Amanda and Lee, Kyungyong and Mou, Lang},
     TITLE = {Scattering diagrams, tight gradings, and generalized
              positivity},
   JOURNAL = {Proc. Natl. Acad. Sci. USA},
  FJOURNAL = {Proceedings of the National Academy of Sciences of the United
              States of America},
    VOLUME = {122},
      YEAR = {\noop{a}2025},
    NUMBER = {18},
     PAGES = {Paper No. e2422893122, 11},
      ISSN = {0027-8424,1091-6490},
   MRCLASS = {13F60 (05A15 16G20)},
  MRNUMBER = {4912377},
       DOI = {10.1073/pnas.2422893122},
       URL = {https://doi.org/10.1073/pnas.2422893122},
}

@misc{BLM25b,
  title={Positivity of generalized cluster scattering diagrams},
  author={Burcroff, Amanda and Lee, Kyungyong and Mou, Lang},
  howpublished={arXiv preprint arXiv:2503.03719},
  year={\noop{b}2025}
}
\bibliographystyle{alpha}
\end{document}